\documentclass[11pt]{article}

 \usepackage{amsmath,amssymb,amscd,amsthm,esint}
 
\usepackage{graphics,amsmath,amssymb,amsthm,mathrsfs}

\usepackage{graphics,amsmath,amssymb,amsthm,mathrsfs,amsfonts}

\usepackage{sidecap}
\usepackage{float}
\usepackage{extarrows}
\usepackage{booktabs}
\usepackage{verbatim}
\usepackage{hyperref}
\usepackage[usenames,dvipsnames]{xcolor}

\newtheorem{thm}{Theorem}[section]
\newtheorem{lemma}[thm]{Lemma}

\theoremstyle{definition}
\newtheorem{remark}[thm]{Remark}

\def\XXint#1#2#3{{\setbox0=\hbox{$#1{#2#3}{\int}$}
         \vcenter{\hbox{$#2#3$}}\kern-.5\wd0}}

\def\R{\mathbb{R}}

\def\e{\varepsilon}

\def\loc{\text{loc}}

\numberwithin{equation}{section}

\begin{document}

\title{Sharp Convergence Rates for  Darcy's Law}

\author{
Zhongwei Shen\thanks{Supported in part by NSF grant DMS-1856235 and by Simons Fellowship.}}
\date{}
\maketitle

\begin{abstract}

This paper is concerned with Darcy's law for an incompressible viscous fluid flowing in a porous medium.
We establish the sharp $O(\sqrt{\e})$ convergence rate in a periodically perforated and
bounded domain in $\R^d$ for $d\ge 2$, where $\e$ represents the size of solid obstacles.
This is achieved by constructing two boundary layer correctors to  control the boundary layers created by the incompressibility 
condition and the discrepancy of boundary values between the solution and the leading term in its asymptotic expansion. 
One of the correctors deals with the tangential boundary data, while the other handles the normal boundary data.
 
\medskip


\medskip

\noindent {\it MR (2020) Subject Classification}: 35Q35;  35B27; 76D07.

\end{abstract}


\section{\bf Introduction}\label{section-1}

This paper is concerned with Darcy's law for an incompressible viscous fluid in a porous medium.
More precisely, we consider the Dirichlet problem for the steady Stokes equations,
\begin{equation}\label{DP}
\left\{
\aligned
-\e^2 \mu  \Delta u_\e + \nabla p_\e  & = f & \quad & \text{ in } \Omega_\e,\\
\text{\rm div} (u_\e) & =0 & \quad & \text{ in } \Omega_\e,\\
u_\e & = 0 & \quad & \text{ on } \partial \Omega_\e,
\endaligned
\right.
\end{equation}
where  $\mu>0$ is the viscosity constant, 
$0< \e< 1$,  and $\Omega_\e$  is  a periodically perforated and  bounded domain  in $\R^d$,   $d\ge 2$.
In \eqref{DP} we have  normalized the velocity vector by a factor $\e^2$,
where $\e$ is the period.
To describe the porous domain $\Omega_\e$, 
we let $Y=[0, 1]^d$ be a closed unit cube and  $Y_s $ (solid part) an open subset of $Y$ with Lipschitz boundary.
Throughout the paper we  shall assume that dist$(\partial Y, \partial Y_s)>0$ and that $Y_f =Y\setminus \overline{Y_s}$ (the fluid part)
is connected.
Let $\Omega$ be a bounded domain in $\R^d$ with Lipschitz boundary.
For $0< \e<1$, define
\begin{equation}\label{O-e}
\Omega_\e = \Omega \setminus \bigcup_k \e \left(\overline{Y_s} +z_k\right),
\end{equation}
where $z_k \in \mathbb{Z}^d$ and the union  is taken over those $k$'s for which $\e(Y+z_k) \subset \Omega$.

For $f\in L^2(\Omega; \R^d)$, let $(u_\e, p_\e)\in H_0^1(\Omega_\e; \R^d) \times L^2(\Omega_\e)$  be the  weak solution of \eqref{DP}
with $\int_{\Omega_\e} p_\e\, dx =0$.
We extend $u_\e$ to the whole domain $\Omega$ by zero and still denote the extension by $u_\e$.
Let $P_\e$ be the extension  of $p_\e$ to $\Omega$, defined by \eqref{P}.
It has been known since late 1970's  that as $\e \to 0$, $u_\e \to u_0$ weakly in $L^2(\Omega; \R^d)$ and
$P_\e \to p_0$ strongly in $L^2(\Omega)$, where $(u_0, p_0)$ is given by a Darcy law,
\begin{equation}\label{D-Law}
\left\{
\aligned
& u_0=  \mu^{-1} K (f-\nabla p_0) & \quad & \text{ in } \Omega,\\
 & \text{\rm div} ( u_0)  =0 & \quad & \text{ in } \Omega,\\
& u_0\cdot n  =0 & \quad &\text{ on } \partial\Omega,
\endaligned
\right.
\end{equation}
with $\int_\Omega p_0\, dx =0$.
In \eqref{D-Law}  the permeability matrix $K=(K_j^i)  $ is a $d\times d$ positive definite and symmetric  matrix defined  by \eqref{K}, and
$n$ denotes the outward unit normal to $\partial\Omega$.
Furthermore, it was observed  in \cite{Allaire-91a} by G. Allaire   that as $\e \to 0$,
\begin{equation}\label{D-Law-1}
 u_\e -\mu^{-1}  W(x/\e) ( f- \nabla p_0) \to 0 \quad \text{ strongly in } L^2(\Omega; \R^d),
\end{equation}
where $W(y)=(W_j^i (y) ) $ is an  1-periodic $d\times d$ matrix  defined by  the cell problem \eqref{W}
and $\fint_Y W(y)\, dy =K$.
For an excellent exposition on Darcy's law and closely  related topics, we refer the reader to \cite{Allaire-1997}
by G. Allaire and A. Mikeli\'{c}.

The purpose of this paper is to study the convergence rates for
$u_\e - \mu^{-1} W(x/\e) (f-\nabla p_0)$ and $P_\e -p_0$ in $L^2(\Omega)$.
The following is the main result of the paper.
The $O(\sqrt{\e})$  rate in \eqref{main-1} is sharp.

\begin{thm}\label{main-theorem-1}
Let $\Omega$ be a bounded $C^{2, \alpha}$ domain in $\R^d$, $d\ge 2$ for some $\alpha>0$.
Also assume that $Y_s$ is an open subset of $Y=[0, 1]^d$ with $C^{1, \alpha }$ boundary. 
Let $(u_\e, p_\e)\in H_0^1(\Omega_\e; \R^d) \times L^2(\Omega_\e)$ be a weak solution of \eqref{DP},
 where $f\in C^{1, 1/2}(\overline{\Omega}; \R^d)$ and $\int_{\Omega_\e} p_\e\, dx =0$.
Then 
\begin{equation}\label{main-1}
\aligned
& \|  u_\e - \mu^{-1}W (x/\e) (f -\nabla p_0)\|_{L^2(\Omega)}
+\| P_\e - p_0 \|_{L^2(\Omega)} \\
 & \qquad \qquad
 + \| \e  \nabla u_\e -\mu^{-1} \nabla W(x/\e)  (f-\nabla p_0)\|_{L^2(\Omega)}
  \le C \sqrt{\e}\,  \| f\|_{C^{1, 1/2} (\Omega)},
\endaligned
\end{equation}
where $C$ depends only on $d$, $\mu$,  $\Omega$, and $Y_s$.
\end{thm}

The first rigorous proof of Darcy's law by homogenization was given by L. Tartar in an appendix of \cite{Sanchez-1980}, using an energy method.
We refer the reader to \cite{Allaire-1997} for references on earlier work on the formal derivation of Darcy's law, using two-scale asymptotic expansions.
In \cite{Allaire-89, Allaire-91a}, the strong convergence of $(u_\e, P_\e)$ in $L^2(\Omega)$
 was established by the method of two-scale convergence.
Also see related work in \cite{LA-1990, Mik-1991, Mik-1994, Mik-1996, Mik-2000, Masmoudi-2004, Lu-2020}.

Regarding the rate of convergence for  $(u_\e, P_\e)$ in $L^2(\Omega)$,
 to the best of the author's knowledge, the only previous result  for a bounded domain with
 the Dirichlet condition was obtained by E. Maru\v{s}i\'{c}-Paloka  and  A. Mikeli\'{c}
in \cite{Mik-1996}, where a rate $O(\e^{1/6})$ was established for the case $d=2$.
See \cite{Mik-1994} for an earlier result for a unbounded domain $\Omega=(0, L) \times \mathbb{R}_+$.
We remark  that for Laplace's equation and systems of linear elasticity,  quantitative error estimates  have been established
in \cite{Lions-1980, Murat-1989, OSY-1992,  Jing-2020, Jing-2020-L}.  
As pointed out in \cite{Mik-1996}, the simple cut-off argument, which seems to work well for standard elliptic 
equations and systems, does not yield any convergence rate for the Stokes equations because of the  incompressibility condition.
In \cite{Mik-1996}, using a stream function from \cite{Temam},
 a boundary layer corrector was constructed in the case $d=2$ to control the boundary layer
near $\partial\Omega$ created by the incompressibility condition.
We mention that \cite{Mik-1996} also treated the case of nonlinear stationary Navier-Stokes equations. 

We now describe our approach to the problem of convergence rates and error estimates, which is based on  energy estimates.
Let
\begin{equation}\label{u-0}
u(x, x/\e) = \mu^{-1} W(x/\e) \big( f(x) -\nabla p_0(x)\big).
\end{equation}
To address the discrepancy of boundary values between $u_\e$ and $u(x, x/\e)$ as well as the incompressibility condition, we introduce two 
boundary layer correctors $(\Psi_t, q_t) $ and $(\Psi_n, q_n) $.
Let $\partial\Omega_\e =\partial\Omega \cup \Gamma_\e$.
The tangential boundary layer  corrector $(\Psi_t, q_t)$ is a weak solution of 
\begin{equation}\label{t-01}
\left\{
\aligned
-\e^2 \mu \Delta\Psi_t +\nabla q_t & =0 & \quad & \text{ in } \Omega_\e,\\
\text{\rm div}(\Psi_t) & =0& \quad & \text{ in }\ \Omega_\e,
\endaligned
\right.
\end{equation}
with boundary data $\Psi_t =0$ on $\Gamma_\e$, and
\begin{equation}\label{t-1a}
\Psi_t = - u(x, x/\e) + [ u(x, x/\e) \cdot n] n \quad \text{ on } \partial\Omega.
\end{equation}
Note that $\Psi_t \cdot n=0$ on $\partial\Omega$.  By the divergence theorem and the Cauchy inequality, this gives,
\begin{equation}\label{t-1b}
 \|\nabla \Psi_t \|^2_{L^2(\Omega_\e)} \le  \| \nabla \Psi_t \|_{L^2(\partial\Omega)} \| \Psi_t\|_{L^2(\partial\Omega)}.
\end{equation}
We  use a localized Rellich estimate in a Lipschitz domain to show that 
\begin{equation}\label{t-1c}
\|\nabla \Psi_t \|_{L^2(\partial\Omega)}
\le C\left\{  \|\nabla_{\tan} \Psi_t \|_{L^2(\partial\Omega)} 
+ \e^{-1/2} \|\nabla \Psi_t \|_{L^2(\Omega_\e)} \right\},
\end{equation}
where $\nabla_{\tan} \Psi_t$ denotes the tangential gradient of $\Psi_t$ on the boundary $\partial\Omega$.
The desired $O(\sqrt{\e})$ bound for $\e \|\nabla \Psi_t \|_{L^2(\Omega_\e)}$ follows  from \eqref{t-1b} and \eqref{t-1c}.
See Section \ref{section-t} for details.

The normal boundary layer corrector $(\Psi_n, q_n)$ is defined as the solution of  the Stokes equations \eqref{t-01} in $\Omega_\e$,
with the boundary conditions $\Psi_n =0$ on $\Gamma_\e$, and 
\begin{equation}\label{n-0}
\Psi_n =- \big[ u(x, x/\e) \cdot n -\gamma \big] n \quad \text{ on } \partial\Omega,
\end{equation}
where 
$$
\gamma =\fint_{\partial\Omega} u(x, x/\e) \cdot n\, d\sigma.
$$
Thanks to \eqref{D-Law}, we may write
\begin{equation}\label{n-0a}
u(x, x/\e)\cdot n = \mu^{-1} n_i\big[ W_j^i (x/\e) -K_j^i \big] \Big( f_j -\frac{\partial p_0}{\partial x_j} \Big)  \quad \text{ on } \partial\Omega
\end{equation}
(the repeated indices are summed from $1$ to $d$).
Furthermore, there exists a 1-periodic tensor $(\phi_{\ell  j}^i)$ such that
\begin{equation}\label{n-0b}
 \phi_{\ell j}^i =-\phi_{ij}^\ell \quad \text{  and } \quad 
W_j^i (y) -K_j^i =\frac{\partial}{\partial y_\ell} \phi_{\ell j}^i (y).
\end{equation}
It follows  from \eqref{n-0a} and \eqref{n-0b} that
\begin{equation}\label{n-0c}
u(x, x/\e)\cdot n
=\e (2\mu)^{-1}  \Big( n_i \frac{\partial}{\partial x_\ell} - n_\ell \frac{\partial}{\partial x_i} \Big) \Big( \phi_{\ell j}^i (x/\e) \Big)  
\cdot \Big( f_j -\frac{\partial p_0}{\partial x_j} \Big)  \quad \text{ on } \partial\Omega.
\end{equation}
Since $n_i \frac{\partial}{\partial x_\ell} - n_\ell \frac{\partial}{\partial x_i}$ is a tangential derivative, 
the formula \eqref{n-0c} allows us  to use an integration by parts on $\partial\Omega$ (see \eqref{IP}), which generates the needed decay factor ${\e}$.
In order to carry out this argument, we use an energy estimate to reduce the problem to the $L^2$ estimate for the Stokes equations
in $\Omega$, whose solutions are then represented by integrals on $\partial\Omega$, using the Poisson kernels.
See Section \ref{section-n} for details.

We point out that  the $C^{1, \alpha}$ condition on $Y_s$ in Theorem \ref{main-theorem-1} is used to ensure the boundedness of $\nabla W$, 
 while the $C^{2, \alpha}$ condition on $\Omega$ is used for the $C^2$ estimates for the Stokes equations in $\Omega$.
 The $C^{1, 1/2}$ condition on $f$ seems to be more or less optimal for the methods used.
 An $O(\sqrt{\e})$ estimate with less regularity on $f$  would be an interesting and challenging problem.

The paper is organized as follows.
In Section \ref{section-p} we introduce some notations and collect several known results  that will be used in later sections.
In Section \ref{section-E} we establish an energy estimate for the Stokes equations in $\Omega_\e$.
The tangential boundary layer corrector $(\Psi_t, q_t)$ is constructed in Section \ref{section-t}, 
while the normal boundary layer  corrector $(\Psi_n, q_n)$ and its estimates are given  in Section \ref{section-n}.
The proof of Theorem \ref{main-theorem-1} is contained  in Section \ref{section-C}, where an interior corrector is constructed.
In fact, a more general case is treated in Section \ref{section-C}, where we assume $u_\e=b \in H^1(\partial\Omega; \R^d)$
on $\partial\Omega$. See Theorem \ref{theorem-C}.
Due to the discrepancy of $u_\e$ and $u(x, x/\e)$ on $\partial\Omega$, 
 the $O(\sqrt{\e}) $ rate  in Theorem \ref{main-theorem-1} is sharp.
 See Remark \ref{last-r}
 
Throughout the paper, the repeated indices are summed from $1$ to $d$.
We will use $C$ and $c$ to denote positive constants that depend at most on $d$, $\mu$,
$\Omega$, and $Y_s$. Since the value of $\mu$ is not relevant in this study,
we  will assume $\mu=1$ in the rest of the paper for simplicity.

\medskip

\noindent 
{\bf Acknowledgement.}
 The author thanks Jinping Zhuge for several valuable comments.
 The author is also grateful for the valuable suggestions and corrections made by the anonymous referees.


\section{Preliminaries}\label{section-p}

Let $Y=[0, 1]^d$ and $Y_s$ (solid part)  be an open subset of $Y$ with Lipschitz boundary.
We assume that dist$(\partial Y, \partial Y_s)>0$ and that (the fluid part) 
$Y_f= Y \setminus \overline{Y_s}$ is connected.
Let
\begin{equation}
\omega =\bigcup_{z\in \mathbb{Z}^d}  \big( Y_f +z\big)
\end{equation}
be the periodic repetition of $Y_f$. It is easy to see that the unbounded domain $\omega$ is connected, 1-periodic,  and that 
$\partial\omega$ is locally Lipschitz.

For $1\le  j \le d$, let $ (W_j (y), \pi_j (y))   = (W_j^1(y), \dots, W_j^d (y), \pi_j (y))\in H^1_{\loc}(\omega; \R^d) \times L^2_{\loc}(\omega)$ be the 1-periodic  solution of
\begin{equation}\label{W}
\left\{
\aligned
-\Delta  W_j  +\nabla \pi_j  & =e_j & \quad & \text{ in } \omega , \\
\text{\rm div}  (W_j) & =0 & \quad & \text{ in }\omega,\\
W_j & =0& \quad & \text{ on } \partial \omega,\\
\endaligned
\right.
\end{equation}
with $\int_{Y_f} \pi_j \, dy=0$, where $e_j=(0, \dots, 1, \dots , 0)$ with $1$  in the $j^{th}$ place.
We extend $W_j$ to $\R^d$ by  zero and define \begin{equation}\label{K}
K_j^i =\int_{Y}  W^i_j (y)\, dy.
\end{equation}
Using
$$
K_j^i =\int_Y \nabla W^\ell_j\cdot \nabla W^\ell_i\, dy,
$$
it  is not hard to show that the $d\times d$ constant matrix $(K_j^i)$  is symmetric and positive definite.

Thanks to the assumption dist$(\partial Y, \partial Y_s)>0$, we have 
$\partial\Omega_\e =\partial \Omega \cup \Gamma_\e$ and 
dist$(\partial\Omega, \Gamma_\e) \ge c \e$, where
\begin{equation}\label{Gamma}
\Gamma_\e = \Omega \cap\partial\Omega_\e\subset  \partial (\e \omega).
\end{equation}
For $f\in L^2(\Omega; \R^d)$, let $(u_\e, p_\e)$ be a weak solution
in $H_0^1(\Omega_\e; \R^d) \times L^2(\Omega_\e)$ of the Dirichlet problem,
\begin{equation}\label{DP-1}
\left\{
\aligned
-\e^2 \Delta u_\e + \nabla p_\e  & = f & \quad & \text{ in } \Omega_\e,\\
\text{\rm div} (u_\e) & =0 & \quad & \text{ in } \Omega_\e,\\
u_\e & =0 & \quad &  \text{ on } \partial \Omega_\e,
\endaligned
\right.
\end{equation}
with $\int_{\Omega_\e} p_\e\, dx=0$.
We extend  $u_\e$ to $\Omega$ by zero and still denote the extension by $u_\e$.
Let  $P_\e$ be the  extension of $p_\e$, defined by
\begin{equation}\label{P}
P_\e (x) =
\left\{ 
\aligned
& p_\e (x)  & \quad & \text{ if } x\in \Omega_\e, \\
& \fint_{\e( Y_f +z_k)} p_\e & \quad & \text{ if } x\in \e (Y_s + z_k)  \text{ and }  \e (Y+z_k) \subset \Omega \text{ for some }
z_k \in \mathbb{Z}^d
\endaligned
\right.
\end{equation}
(see \cite{LA-1990}).

\begin{thm}\label{homo-thm}
Let $\Omega$ be a bounded Lipschitz domain in $\R^d$, $d\ge 2$.
Let $p_0\in H^1(\Omega)$ be the  weak solution of the Neumann problem,
\begin{equation}\label{homo-2}
\left\{
\aligned
\frac{\partial}{\partial x_i} K_j^i \Big (f_j -\frac{\partial p_0}{\partial x_j}\Big) & =0 & \quad  & \text{ in } \Omega,\\
n_i K_j^i \Big(f_j -\frac{\partial p_0}{\partial x_j} \Big) & =0 & \quad & \text{ on } \partial\Omega,
\endaligned
\right.
\end{equation}
with $\int_{\Omega} p_0\, dx=0$,
where $n=(n_1, \dots, n_d)$ denotes the outward unit normal to $\partial\Omega$.
Then, as  $\e \to 0$, 
\begin{equation}\label{homo-1}
\left\{
\aligned
{u}_\e - W_j (x/\e) \Big ( f_j - \frac{\partial p_0}{\partial x_j}\Big) & \to 0 & \quad & \text{ in }L^2(\Omega; \R^d), \\ 
P_\e  -p_0 & \to 0 & \quad & \text{ in } L^2(\Omega).
\endaligned
\right.
\end{equation}
  \end{thm}

As indicated in Introduction, a proof of Theorem \ref{homo-thm}, using the method of two-scale convergence,
may be found in \cite{Allaire-89, Allaire-91a}.
We do not use the theorem in this paper.
However, we will need several other  known results stated below.

\begin{lemma}\label{P-lemma}
Let $\Omega$ be a bounded Lipschitz domain in $\R^d$, $d\ge 2$.
Assume that $\Gamma_\e \neq \emptyset$.
Let $u \in H^1(\Omega_\e)$ with  $u=0$ on $\Gamma_\e$. Then
\begin{equation}\label{PI}
\| u \|_{L^2(\Omega_\e)} \le C\e  \| \nabla u \|_{L^2(\Omega_\e)}.
\end{equation}
\end{lemma}

\begin{proof}
The case $u\in H_0^1(\Omega_\e)$
 is more or less well known.
See e.g. \cite{CPS-2007}.
The proof for the case $u\in H^1(\Omega_\e)$ with $u_\e=0$ on $\Gamma_\e$ is the same.
 We sketch a proof here  for the reader's convenience.
Suppose $\e(Y+z_k)\subset\Omega$ for some $z_k\in \mathbb{Z}^d$.
Since $u=0$ on $\Gamma_\e$, it follows by Poincar\'e's inequality that
\begin{equation}\label{PI-1}
\int_{\e (Y_f +z_k)} |u|^2\, dx \le C \e^2 \int_{\e (Y_f +z_k)} |\nabla u|^2\, dx.
\end{equation}
Similarly, 
\begin{equation}\label{PI-2}
\int_{B(x_0, C \e)\cap \Omega_\e} |u|^2\, dx 
\le C \e^2 \int_{B(x_0, C\e)\cap \Omega_\e} |\nabla u|^2\, dx,
\end{equation}
if  $x_0\in \partial\Omega$ and $\e (Y+ z) \subset B(x_0, C\e) \cap \Omega$ for some $z\in \mathbb{Z}^d$.
The estimate \eqref{PI} follows from \eqref{PI-1} -\eqref{PI-2} by a covering argument.
\end{proof}

 \begin{lemma}\label{lemma-R}
 Let $\Omega$ be a bounded Lipschitz domain in $\R^d$, $d\ge 2$.
 There exists a bounded linear operator
 \begin{equation}\label{R}
 R_\e: H^1(\Omega; \R^d) \to H^1(\Omega_\e, \R^d),
 \end{equation}
 such that
 \begin{equation}\label{R-1}
 \left\{
 \aligned
 & R_\e (u)=0  \quad \text{ on } \Gamma_\e \quad \text{ and } \quad R_\e(u)=u \quad \text{ on } \partial\Omega,\\
 & R_\e (u)\in H^1_0(\Omega_\e; \R^d) \quad \text{ if } u\in H^1_0(\Omega; \R^d),\\
 &R_\e (u) =u \quad \text{ in } \Omega_\e\  \text{ if } \ u=0 \text{ on } \Gamma_\e,\\
 & \text{\rm div}(R_\e (u) ) =0 \quad \text{ in } \Omega_\e\quad  \text{ if } \ \text{\rm div}(u) =0
  \text{ in }\Omega, 
 \endaligned
 \right.
 \end{equation}
 and
 \begin{equation}\label{R-2}
 \e\, \| \nabla R_\e (u)\|_{L^2(\Omega_\e)}
 + \| R_\e (u)\|_{L^2(\Omega_\e)}
 \le C \Big\{
 \e \| \nabla u\|_{L^2(\Omega)} + \| u\|_{L^2(\Omega)} \Big\},
 \end{equation}
 where $C$ depends only on $\Omega$ and $Y_s$.
 Moreover, 
 \begin{equation}\label{R-3}
 \|\text{\rm div} (R_\e (u) ) \|_{L^2(\Omega_\e)}
 \le C\,  \| \text{\rm div} (u)\|_{L^2(\Omega)}.
 \end{equation}
 \end{lemma}
 
 \begin{proof}
 The proof is the similar to  that of a lemma due to Tartar (in an appendix of \cite{Sanchez-1980}, also see Lemma 1.7 in \cite{Allaire-1997}).
 Let $u\in H^1(\Omega; \R^d)$.
 For each $\e (Y+z) \subset \Omega$, where $z\in \mathbb{Z}^d$,  we define $R_\e(u)$ on $\e (Y_f+z)$ by the Dirichlet problem, 
 \begin{equation}\label{Res}
 \left\{
 \aligned
 -\e^2 \Delta R_\e (u) +\nabla q & =-\e^2 \Delta u & \quad & \text{ in } \e (Y_f +z),\\
 \text{\rm div} (R_\e (u)) & = \text{\rm div} (u)
 + \frac{1}{|\e (Y_f + z)|}
 \int_{\e (Y_s +z)} \text{\rm div} (u)\, dx &\quad & \text{ in } \e (Y_f +z),\\
 R_\e (u) & =0& \quad & \text{ on } \partial ( \e (Y_s +z)),\\
 R_\e (u) & = u & \quad & \text{ on } \partial (\e (Y+z)).
 \endaligned
 \right.
 \end{equation}
 If $x\in \Omega_\e$ and $x  \notin \e (Y_f +z) $ for any $\e(Y+z)\subset \Omega$,
 we let $R_\e (u)(x) =u(x)$.
 It is not hard to show that $R_\e(u)\in H^1(\Omega_\e; \R^d)$ satisfies the conditions in  \eqref{R-1}-\eqref{R-3}.
 \end{proof}
 
 \begin{lemma}\label{div-lemma}
 Let $\Omega$ be a bounded Lipschitz domain in $\R^d$, $d\ge 2$.
 Suppose that  $g \in L^2(\Omega_\e)$ and $\int_{\Omega_\e}g\, dx =0$. Then there exists $v_\e\in H_0^1(\Omega_\e; \R^d)$ such that
$
\text{\rm div} (v_\e) = g  \text{ in } \Omega_\e
$
and
\begin{equation}\label{div-1}
 \e \| \nabla v_\e \|_{L^2(\Omega_\e)} + \| v_\e\|_{L^2(\Omega_\e)}
 \le C \| g\|_{L^2(\Omega_\e)},
 \end{equation}
 where $C$ depends only on $\Omega$ and $Y_s$.
 \end{lemma}
 
 \begin{proof}
 See e.g. \cite[pp.146-148]{CPS-2007}.
 \end{proof}
 
 We end this section with some observations on the rescaled solutions  
 $(W_j (x/\e), \e \pi_j (x/\e))$ in $\Omega_\e$.  It follows from  \eqref{W} by rescaling that
 \begin{equation}\label{w-n-0}
 \left\{
 \aligned
 -\e^2 \Delta \big\{ W_j(x/\e) \big\} + \nabla \big\{ \e \pi_j(x/\e)  \big\}  & = e_j & \quad & \text{ in } \e \omega,\\
 \text{\rm div} (W_j (x/\e)) & =0 & \quad & \text{ in } \e \omega,\\
 W_j (x/\e) & =0 & \quad & \text{ on } \partial(\e\omega).
 \endaligned
 \right.
 \end{equation}
 We extend both $W_j$ and $\pi_j$ to $\R^d$ by zero.
Clearly, 
 \begin{equation}\label{W-n-1}
 \left\{
 \aligned
 \text{\rm div} (W_j (x/\e)) & =0 & \quad & \text{ in } \R^d,\\
 W_j (x/\e) & =0 & \quad & \text{ on } \Gamma_\e =\Omega \cap \partial\Omega_\e.
 \endaligned
 \right.
\end{equation}
Note that in the construction of $\Omega_\e$, the holes near $\partial\Omega$ are not removed.
As a result, the first equation in \eqref{w-n-0} needs to be modified for $\Omega_\e$. In fact,
a computation using integration by parts  shows that
\begin{equation}\label{W-n-2}
-\e^2 \Delta \big\{ W_j (x/\e) \big\}
+ \nabla\big\{ \e \pi_j (x/\e)  \big\}=e_j+ \sigma_{\e, j} \quad \text{ in } \Omega_\e,
\end{equation}
where $\sigma_{\e, j}\in H^{-1}(\Omega_\e; \R^d)$ is given by
\begin{equation}\label{W-n-3}
\aligned
& \langle \sigma_{\e, j}, \psi \rangle_{H^{-1}(\Omega_\e) \times H^1_0(\Omega_\e)} \\
& =\sum_k  \left\{ -\int_{\Omega \cap \e (Y_s+z_k)} e_j \cdot \psi\, dx
-\e  \int_{\overline{\Omega} \cap \partial (\e (Y_s+z_k))}
 (\nabla W_j (x/\e) n - \pi_j (x/\e) n) \cdot \psi  d\sigma \right\}.
\endaligned
\end{equation}
The sum in \eqref{W-n-3} is taken over those $k$'s for which $z_k \in \mathbb{Z}^d$
and $\e (Y + z_k) \cap \partial\Omega \neq \emptyset$,
and $n$ denotes the outward unit normal.
Under the assumption that $\partial Y_s$ is $C^{1, \alpha}$,
it is known that $|\nabla W_j|$ and $\pi_j$ are bounded in $\R^d$.
It follows that if $g \in H^1_{loc} (\R^d)$ and $\psi \in H_0^1(\Omega_\e; \R^d)$, then
$$
\aligned
&| \langle \sigma_{\e, j}, g \psi \rangle_{H^{-1}(\Omega_\e) \times H^1_0(\Omega_\e)} |
\le C 
 \sum_k \left\{ \int_{\e (Y+z_k)} |g \psi|\, dx
+  \e  \int_{\partial (\e (Y_s +z_k))} |g \psi|\, d\sigma \right\}
\endaligned
$$
$(\psi$ is extended to $\R^d$ by zero).
Using the inequality
$$
\int_{\partial (\e(Y_s +z_k))}
|u|^2\, d\sigma
\le C \e \int_{\e (Y+z_k)} |\nabla u |^2\, dx
+C \e^{-1} \int_{\e (Y+z_k)} | u|^2\, dx,
$$
 \eqref{PI} and  the Cauchy  inequality,
  one may prove that
  \begin{equation}\label{W-n-4}
  | \langle \sigma_{\e, j}, g \psi \rangle_{H^{-1}(\Omega_\e) \times H^1_0(\Omega_\e)} |
\le C \e \big\{ \e \| \nabla g \|_{L^2(\Sigma_{c\e} )} + \| g \|_{L^2(\Sigma_{c\e})} \big\}  \| \nabla \psi \|_{L^2(\Omega_\e)}
  \end{equation}
  for any $\psi \in H_0^1(\Omega_\e; \R^d)$, where $\Sigma_{c \e} = \{ x\in \R^d: \text{\rm dist}(x, \partial\Omega)< c\e \}$.
  
 
 \section{Energy estimates}\label{section-E}
 
In this section we establish  the energy estimates for the Dirichlet problem, 
\begin{equation}\label{DP-E}
\left\{
\aligned
-\e^2 \Delta u_\e + \nabla p_\e  & = f +  \e\,  \text{\rm div} (F) & \quad & \text{ in } \Omega_\e,\\
\text{\rm div}(u_\e) & =g & \quad & \text{ in } \Omega_\e,\\
u_\e & =0 & \quad & \text{ on } \Gamma_\e, \\
u_\e & =h & \quad & \text{ on } \partial\Omega,
\endaligned
\right.
\end{equation}
where $(g, h)$ satisfies the compatibility condition,
\begin{equation}\label{c-p}
\int_\Omega g\, dx =\int_{\partial\Omega} h \cdot n\, d\sigma.
\end{equation}

\begin{thm}\label{theorem-E}
Let $\Omega$ be a bounded  domain in $\R^d$, $d\ge 2$ with Lipschitz boundary.
Let $(u_\e, p_\e) \in H^1(\Omega_\e; \R^d) \times L^2(\Omega_\e)$ be a weak solution of
\eqref{DP-E} with $\int_{\Omega_\e} p_\e\, dx=0$.
Then
\begin{equation}\label{E-1}
\aligned
 & \e \|\nabla u_\e \|_{L^2(\Omega_\e )} 
+ \| u_\e \|_{L^2(\Omega_\e)}
+ \| p_\e \|_{L^2(\Omega_\e)}\\
&
\le C \left\{ \| f\|_{L^2(\Omega_\e)}
+ \| F\|_{L^2(\Omega_\e)}
+ \| g\|_{L^2(\Omega_\e)}
+ \| h\|_{L^2(\partial\Omega)}
+ \e \| h\|_{H^{1/2}(\partial\Omega)}
\right\},
\endaligned
\end{equation}
for any $0< \e< 1$,
where $C$ depends only on $\Omega$ and $Y_s$.
\end{thm}

\begin{proof}
We divide the proof into several steps.

\noindent Step 1. 
By Lemma \ref{div-lemma}, there exists $v_\e\in H_0^1(\Omega_\e; \R^d)$ such that
$\text{div} (v_\e) =p_\e$ in $\Omega_\e$ and
\begin{equation}\label{E-10}
\e \|\nabla v_\e\|_{L^2(\Omega_\e)} + \| v_\e\|_{L^2(\Omega_\e)} 
\le C \| p_\e\|_{L^2(\Omega_\e)}.
\end{equation}
By using $v_\e$ as a test  function we see that
$$
\aligned
\| p_\e\|_{L^2(\Omega_\e)}^2
 &\le  \e^2 \| \nabla u_\e\|_{L^2(\Omega_\e)}  \| \nabla v_\e\|_{L^2(\Omega_\e)}
+ \| f \|_{L^2(\Omega_\e)} \| v_\e\|_{L^2(\Omega_\e)}
+ \e \| F \|_{L^2(\Omega_\e)} \|\nabla v_\e\|_{L^2(\Omega_\e)}\\
&\le C \| p_\e\|_{L^2(\Omega)}
\left\{  \e \|\nabla u_\e\|_{L^2(\Omega_\e)}
+ \| f\|_{L^2(\Omega_\e)}
+ \| F \|_{L^2(\Omega_\e)} \right\},
\endaligned
$$ 
where we have used \eqref{E-10} for the last inequality.
This gives
\begin{equation}\label{E-10-0}
\| p_\e \|_{L^2(\Omega_\e)}
\le C \left\{  \e \|\nabla u_\e\|_{L^2(\Omega_\e)}
+ \| f\|_{L^2(\Omega_\e)}
+ \| F \|_{L^2(\Omega_\e)} \right\}.
\end{equation}

\noindent Step 2. We consider the case $h=0$ on $\partial\Omega$.
This allows us to use the test function $u_\e\in H^1_0(\Omega_\e; \R^d)$  to obtain 
$$
\e^2 \|\nabla u_\e\|_{L^2(\Omega_\e)}^2
\le  \| p_\e \|_{L^2(\Omega_\e)} \| g\|_{L^2(\Omega_\e)}
+  \| f\|_{L^2(\Omega_\e)} \| u_\e\|_{L^2(\Omega_\e)}
+ \e \| F \|_{L^2(\Omega_\e)} \|\nabla u_\e\|_{L^2(\Omega)},
$$
where we have also used the Cauchy inequality.
It follows from the inequality $\| u_\e\|_{L^2(\Omega_\e)}
\le C \e \| \nabla u_\e\|_{L^2(\Omega_\e)}$ as well as the Cauchy  inequality that
\begin{equation}\label{E-9}
\e^2 \|\nabla u_\e\|_{L^2(\Omega_\e)}^2
\le C\left\{  \| p_\e \|_{L^2(\Omega_\e)} \| g \|_{L^2(\Omega_\e)}
+  \| F \|_{L^2(\Omega_\e)}^2 + \| f\|_{L^2(\Omega_\e)}^2 \right\}.
\end{equation}
This, together with \eqref{E-10-0}, yields \eqref{E-1} by the Cauchy inequality.

\medskip

\noindent Step 3. In the general case, we let $(H, q)\in H^1(\Omega; \R^d) \times L^2(\Omega)$ be a weak solution of
$$
-\Delta H +\nabla q =0 \quad \text{ and } \quad \text{\rm div}(H) =\gamma \quad \text{ in } \Omega,
$$
with boundary data $H=h$ on $\partial\Omega$, where 
$$
\gamma = \frac{1}{|\Omega|} \int_{\partial\Omega} h \cdot n \, d\sigma.
$$
Let $w_\e= R_\e (H)$, where $R_\e$ is the operator given by Lemma \ref{lemma-R}.
Note that $w_\e=0 $ on $\Gamma_\e$, $w_\e =h$ on $\partial \Omega$, 
\begin{equation}\label{E-12}
\aligned
\e \| \nabla w_\e \|_{L^2(\Omega_\e)} + \| w_\e\|_{L^2(\Omega_\e)}
 & \le C \left\{ \e \| \nabla H\|_{L^2(\Omega)} + \| H\|_{L^2(\Omega)} \right\},\\
 \endaligned
 \end{equation}
 and
 \begin{equation}\label{E-13}
 \| \text{\rm div}(w_\e)\|_{L^2(\Omega_\e )}  \le C |\gamma|.
\end{equation}
Thus, $u_\e -w_\e\in H^1_0(\Omega_\e; \R^d)$, and 
$$
-\e^2 \Delta (u_\e -w_\e) +\nabla p_\e = f + \e\,  \text{\rm div} (F) +\e^2 \Delta w_\e \quad \text{ in } \Omega_\e.
$$
Hence, by Step 2, we obtain
$$
\aligned
& \e  \| \nabla (u_\e -w_\e ) \|_{L^2(\Omega_\e)} 
+ \| u_\e -w_\e\|_{L^2(\Omega_\e)}
+ \| p_\e \|_{L^2(\Omega_\e)}\\
&\le C \big\{ \| f\|_{L^2(\Omega_\e)} + \| F\|_{L^2(\Omega_\e)} + \e \|\nabla w_\e\|_{L^2(\Omega_\e)}
+ \| g\|_{L^2(\Omega_\e)} +  |\gamma|   \big\},
\endaligned
$$
where we have used \eqref{E-13}.
It follows from \eqref{E-12}  that
\begin{equation}\label{E-14}
\aligned
 & \e \| \nabla u_\e\|_{L^2(\Omega_\e)} + \| u_\e\|_{L^2(\Omega_\e)} + \| p_\e\|_{L^2(\Omega_\e)}\\
&\le C \big\{ \| f\|_{L^2(\Omega_\e)} + \| F\|_{L^2(\Omega_\e)} 
+ \| g\|_{L^2(\Omega_\e)} +  |\gamma|
+ \e \| \nabla H \|_{L^2(\Omega)} + \| H \|_{L^2(\Omega)}  \big\}.
\endaligned
\end{equation}

\noindent Step 4. To estimate $\|\nabla H\|_{L^2(\Omega)}$ and $\| H \|_{L^2(\Omega)}$, 
we let 
$$
\widetilde{H}= H - \gamma d^{-1} (x-x_0),
$$
where $x_0\in \Omega$ is  fixed.
Note that 
$$
-\Delta \widetilde{H} +\nabla q =0 \quad \text{ and } \quad \text{\rm div} (\widetilde{H} )=0 \quad 
\text{ in } \Omega.
$$
By the energy estimates,
$$
\|\nabla \widetilde{H}\|_{L^2(\Omega)}
\le C \| \widetilde{H}\|_{H^{1/2}(\partial \Omega)}
\le C \| h \|_{H^{1/2}(\partial\Omega)},
$$
and by the nontangential-maximal-function estimates for the Stokes equations in \cite{FKV},
$$
\| \widetilde{H}\|_{L^2(\Omega)} \le C \| \widetilde{H} \|_{L^2(\partial\Omega)}
\le C \| h\|_{L^2(\partial\Omega)}.
$$
It follows that
$$
\e \| \nabla H \|_{L^2(\Omega)} 
+ \| H \|_{L^2(\Omega)}
\le C \big\{ \e \| h\|_{H^{1/2}(\partial\Omega)} + \| h\|_{L^2(\partial\Omega)} \big\}.
$$
This, together with \eqref{E-14}, completes the proof.
\end{proof}

\begin{remark}\label{remark-N}
{\rm
If we replace the right-hand side $ f +\e\, \text{\rm div}(F)$ of the first equation in \eqref{DP-E} by some
$\sigma_\e \in H^{-1}(\Omega_\e; \R^d)$ that satisfies the condition 
$$
|\langle \sigma_\e, \psi \rangle_{H^{-1}(\Omega_\e) \times H_0^1(\Omega_\e)}|
\le \e N \|\nabla \psi \|_{L^2(\Omega_\e)}
$$
for any $\psi \in H_0^1(\Omega_\e; \R^d)$ and  some $N=N(\sigma_\e)>0$, then the same argument as in the proof of Theorem \ref{theorem-E} gives
\begin{equation}\label{re-N-1}
\e \| \nabla u_\e \|_{L^2(\Omega_\e)}
+ \| u_\e\|_{L^2(\Omega_\e)}
+ \| p_\e\|_{L^2(\Omega_\e)}
\le C \Big\{ N + \| g \|_{L^2(\Omega_\e)}
+ \| h\|_{L^2(\partial\Omega)}
+ \e \| h \|_{H^{1/2} (\partial\Omega)} \Big\}
\end{equation}
for any $0< \e<1$, where $C$ depends only on $\Omega$ and $Y_s$.
Note that if $\sigma_\e = f+\e\,  \text{\rm div} (F)$, then $N = C \left\{ \| f\|_{L^2(\Omega_\e) }+ \| F \|_{L^2(\Omega_\e)} \right\}$.
}
\end{remark}


\section{Correctors for tangential boundary data}\label{section-t}

Consider the Dirichlet problem,
\begin{equation}\label{DP-3}
\left\{
\aligned
-\e^2 \Delta u_\e + \nabla p_\e  & = 0 & \quad & \text{ in } \Omega_\e,\\
\text{\rm div} (u_\e) & =0 & \quad & \text{ in } \Omega_\e,\\
u_\e & =0 & \quad &  \text{ on } \Gamma_\e,\\
u_\e & = h & \quad & \text{ on } \partial \Omega,
\endaligned
\right.
\end{equation}
with boundary data $h$ satisfying the condition 
\begin{equation}\label{BD}
h \cdot n =0 \quad \text{ on } \partial\Omega.
\end{equation}
The goal of this section is to prove the following.

\begin{thm}\label{theorem-t}
Let $\Omega$ be a bounded Lipschitz domain in $\R^d$, $d\ge 2$.
Let $(u_\e, p_\e)$ be a weak solution in $H^1(\Omega_\e; \R^d) \times L^2(\Omega_\e)$  of \eqref{DP-3} with $\int_{\Omega_\e} p_\e\, dx =0$,
where $h\in H^1(\partial\Omega; \R^d)$ satisfies \eqref{BD}.
Then
\begin{equation}\label{t-1}
\e \|\nabla u_\e\|_{L^2(\Omega_\e)}
+\| u_\e\|_{L^2(\Omega_\e)}
+\| p_\e  \|_{L^2(\Omega_\e)} 
\le C \sqrt{\e}
\left\{ \| h \|_{L^2(\partial\Omega)}
+\e\| \nabla_{\tan} h \|_{L^2(\partial\Omega)} \right\},
\end{equation}
where $\nabla_{\tan} h$ denotes the tangential gradient of $h$ on $\partial\Omega$.
\end{thm}

Let
\begin{equation}\label{D}
\aligned
D_r & =\big\{ (x^\prime, x_d)\in \R^d: \  |x^\prime|< r \ \ \text{ and } \ \ \psi(x^\prime)< x_d < 100 d(M+1)r \big\},\\
I_r  & =\big\{ (x^\prime, \psi (x^\prime) )\in \R^d: \ |x^\prime|< r \big\},
\endaligned
\end{equation}
where $\psi: \R^{d-1} \to \R$ is a Lipschitz function such that
$\psi (0)=0$ and $\|\nabla \psi\|_\infty \le M$.

\begin{lemma}\label{lemma-t-1}
Let $(v, q)$ be a weak solution in $H^1(D_r; \R^d)\times L^2(D_r)$ of the Dirichlet problem,
\begin{equation}\label{DP-local}
\left\{
\aligned
-\Delta v +\nabla q & =0&\quad &\text{ in } D_r,\\
\text{\rm div} (v) & =0 & \quad & \text{ in } D_r,\\
v & =g & \quad & \text{ on } \partial D_r,
\endaligned
\right.
\end{equation}
where $0< r<\infty$ and $g\in H^1(\partial D_r; \R^d)$ satisfies the condition $\int_{\partial D_r} g \cdot n\, d\sigma =0$.
Then
\begin{equation}\label{t-1-1}
\int_{\partial D_r} |\nabla v|^2\, d\sigma 
\le C \int_{\partial D_r} |\nabla_{\tan} v|^2\, d\sigma,
\end{equation}
where $C$ depends only on $d$ and $M$.
\end{lemma}

\begin{proof}
By dilation we may assume $r=1$, in which case 
the Rellich estimate \eqref{t-1-1} was proved in \cite[Theorem 4.15]{FKV}.
\end{proof}

\begin{lemma}\label{lemma-t-2}
Let $(v, q)$ be a weak solution of \eqref{DP-local} with $r=2$.
Then
\begin{equation}\label{t-2-1}
\int_{I_1} |\nabla v|^2\, d\sigma \le  C\int_{I_2} |\nabla_{\tan}  g |^2\, d\sigma + C \int_{D_2} |\nabla v|^2\, dx,
\end{equation}
where $C$ depends only on $d$ and $M$.
\end{lemma}

\begin{proof}
It follows from \eqref{t-1-1} that  for $1< r< 2$,
$$
\int_{I_1} |\nabla v|^2\, d\sigma
  \le C \int_{\partial D_r} |\nabla v|^2\, d\sigma 
\le C \int_{\partial D_r} |\nabla_{\tan}  v|^2\, d\sigma.
$$
Hence,
$$
\int_{I_1} |\nabla v|^2\, d\sigma \le C \int_{I_2} |\nabla_{\tan} g |^2\, d\sigma
+ C \int_{D_2 \cap \partial D_r} |\nabla v|^2\, d\sigma.
$$
By integrating the inequality  above  in $r$ over the interval $(1, 2)$, we obtain \eqref{t-2-1}.
\end{proof}

\begin{proof}[Proof of Theorem \ref{theorem-t}]

We start with the observation,
$$
\e^2 \int_{\Omega_\e} |\nabla u_\e|^2\, dx = \e^2 \int_{\partial\Omega} \frac{\partial u_\e}{\partial n} \cdot u_\e\, d\sigma,
$$
where we have used \eqref{DP-3} and \eqref{BD}.
It follows by the Cauchy inequality that
\begin{equation}\label{t-2}
\e^2 \int_{\Omega_\e} |\nabla u_\e|^2\, dx
\le \e^2 \| \nabla u_\e\|_{L^2(\partial\Omega)} \| h\|_{L^2(\partial\Omega)}.
\end{equation}
We will show that
\begin{equation}\label{t-3}
\int_{\partial\Omega} |\nabla u_\e|^2\, d\sigma
\le C \int_{\partial\Omega} |\nabla_{\tan}  h|^2\, d\sigma + \frac{C}{\e} \int_{\Sigma_{c\e}} |\nabla u_\e|^2\, dx,
\end{equation}
where $\Sigma_{c\e}  =\{ x\in \Omega: \ \text{dist}(x, \partial\Omega)< c\e\}\subset \Omega_\e$.
Assume  \eqref{t-3} for a moment.
Then
$$
\| \nabla u_\e \|_{L^2(\partial\Omega)}
\le C\| \nabla_{\tan}  h\|_{L^2(\partial\Omega)} + C \e^{-1/2} \|\nabla u_\e\|_{L^2(\Omega_\e)}.
$$
This, together with \eqref{t-2} and the Cauchy inequality, gives
$$
\aligned
\e^2 \| \nabla u_\e\|_{L^2(\Omega_\e )}^2
&\le C \e^2\| h\|_{L^2(\partial\Omega)} \| \nabla_{\tan}  h\|_{L^2(\partial\Omega)}
+ C \e^{3/2} \| h\|_{L^2(\partial\Omega)} \| \nabla u_\e\|_{L^2(\Omega_\e)}\\
&\le  C \e^2 \| h\|_{L^2(\partial\Omega)} \| \nabla_{\tan}  h\|_{L^2(\partial\Omega)}
+ C \e \| h\|_{L^2(\partial\Omega)}^2
+ (1/2)   \e^2 \|\nabla u_\e\|^2_{L^2(\Omega_e)},
\endaligned
$$
which yields  the estimate for $\e \| \nabla u_\e \|_{L^2(\Omega_\e)} $ in  \eqref{t-1}.
The estimate for $\|u_\e\|_{L^2(\Omega_\e)}$ follows by \eqref{PI}, while
the  bound  for $\|p_\e\|_{L^2(\Omega_\e)}  $ follows from \eqref{E-10-0}.

It remains to prove \eqref{t-3}.
To this end, we shall prove in a few lines below  that
\begin{equation}\label{t-4}
\int_{B(x_0, c_0\e)\cap \partial \Omega} |\nabla u_\e|^2\, d\sigma
\le C \int_{B(x_0, c_1 \e)\cap \partial \Omega} |\nabla_{\tan}  h|^2\, d\sigma
+ \frac{C}{\e} \int_{B(x_0, c_1\e)\cap \Omega} |\nabla u_\e|^2\, dx
\end{equation} 
for any  $x_0\in \partial \Omega$, where $0<c_0< c_1$ are sufficiently small.
The desired estimate \eqref{t-3} follows from \eqref{t-4} by covering $\partial\Omega$ with a finite number of balls
$\{B(x_k, c_0 \e)\}$ centered on $\partial\Omega$.

Finally, we note that if $v (x)=u_\e (\e x)$ and $q(x)=\e^{-1} p_\e (\e x)$,
then $-\Delta v +\nabla q=0$ and div$(v)=0$.
As a result, the estimate \eqref{t-4} follows from \eqref{t-2-1} by a translation and rotation of
the coordinate system. We point out that since the constant $C$ in \eqref{t-2-1} depends only on $d$ and $M$,
the constant $C$ in \eqref{t-4} depends only on $d$ and the Lipschitz character of $\Omega$.
In particular, $C$ does not depend on $\e$.
\end{proof}

As a corollary of Theorem \ref{theorem-t}, we are able to construct  a tangential boundary layer corrector.

\begin{thm}\label{t-c-theorem}
Let $\Omega$ be bounded domain with $C^{2, \alpha}$ boundary for some $\alpha>0$.
Also assume that $\partial Y_s$ is $C^{1, \alpha}$.
Let $(\Psi_t, q_t)$ be a weak solution in $H^1(\Omega_\e; \R^d) \times L^2(\Omega_\e)$ of the Dirichlet problem 
\eqref{DP-3} with $\int_{\Omega_\e} q_t \, dx=0$, where the boundary data $h$ is given by
\begin{equation}\label{c-t-1}
h= b -W_j(x/\e) \Big(f_j -\frac{\partial p_0}{\partial x_j}\Big) 
+ \Big[-b\cdot n  + n_i W^i_j (x/\e) \Big(f_j -\frac{\partial p_0}{\partial x_j}\Big)  \Big] n,
\end{equation}
$ f\in C^{1, 1/2} (\overline{\Omega}; \R^d)$, 
$b\in H^1(\partial\Omega; \R^d)$ satisfies $\int_{\partial\Omega} b \cdot n \, d \sigma =0$,  and $p_0$ is the solution of the Neumann problem,
\begin{equation}\label{p-0}
\left\{
\aligned
K_j^i \frac{\partial}{\partial x_i} \Big( f_j -\frac{\partial p_0}{\partial x_j} \Big)  & = 0 & \quad & \text{ in } \Omega,\\
n_i K_j^i \Big(f_j -\frac{\partial p_0}{\partial x_j}\Big)  & =b\cdot n & \quad & \text{ on } \partial\Omega,
\endaligned
\right.
\end{equation}
with $\int_\Omega p_0\, dx=0$.
Then
\begin{equation}\label{c-t-2}
\aligned
 & \e \|\nabla \Psi_t \|_{L^2 (\Omega_\e)} + \|\Psi_t \|_{L^2(\Omega_\e)} +  \| q_t  \|_{L^2(\Omega_\e)} \\
& \le C \sqrt{\e} \Big\{ \| f-\nabla p_0\|_{L^2(\partial\Omega)} + \| b \|_{L^2(\partial\Omega)} 
+ \e  \| \nabla_{\tan} ( f-\nabla p_0)\|_{L^2(\partial\Omega)}  + \e \|\nabla_{\tan} b \|_{L^2(\partial\Omega)} \Big\} 
\endaligned
\end{equation}
for any $0< \e< 1$.
\end{thm}

\begin{proof}
Note that $h\cdot n=0$ on $\partial\Omega$. Also,
under the assumption that $\partial Y_s$ is $C^{1, \alpha }$, we have $W_j=W_j(y) \in C^1(\overline{\omega})$.
It follows that
$$
\| h\|_{L^2(\partial\Omega)}
\le C \big\{  \| f -\nabla p_0\|_{L^2(\partial\Omega)} + \| b\|_{L^2(\partial\Omega)} \big\},
$$
and
$$
\aligned
& \|\nabla_{\tan} h \|_{L^2(\partial\Omega)}\\
& \le C \Big\{   \e^{-1} \| f-\nabla p_0 \|_{L^2(\partial\Omega)} +  \| \nabla_{\tan} ( f-\nabla p_0) \|_{L^2(\partial\Omega)}
+ \| \nabla_{\tan} b \|_{L^2(\partial\Omega)}   +\|b\|_{L^2(\partial \Omega)} \Big\}.
\endaligned
$$
As a result, the estimate \eqref{c-t-2} follows readily from \eqref{t-1}.
\end{proof}


\section{Correctors for  normal boundary data}\label{section-n}

In this section we consider the Dirichlet problem \eqref{DP-3}, where the boundary data $h$ is given by
\begin{equation}\label{N}
h=\big\{ n_i [ W_j^i  (x/\e) - K_j^i ] g_j  -\gamma \big\} n,
\end{equation}
where $g=(g_1, g_2, \dots, g_d)  \in H^1(\partial\Omega; \R^d)$, and $\gamma\in \R$ is chosen so that
$\int_{\partial \Omega} h \cdot n \, d\sigma =0$, i.e.,
\begin{equation}\label{gamma}
\gamma=\frac{1}{|\partial\Omega|}
\int_{\partial\Omega}
n_i  [ W_j^i (x/\e) - K_j^i ] g_j\, d\sigma.
\end{equation}
The goal of this section is to prove the following.

\begin{thm}\label{theorem-N}
Let $\Omega$ be a bounded $C^{2, \alpha}$ domain in $\R^d$ for some $\alpha>0$.
Also assume that $\partial Y_s$ is $C^{1, \alpha}$.
Let $(u_\e, p_\e)$ be a weak solution in $H^1(\Omega_\e; \R^d) \times L^2(\Omega_\e)$ of 
\eqref{DP-3} with $\int_{\Omega_\e} p_\e\, dx =0$, where $h$ is given by \eqref{N}.
Then
\begin{equation}\label{N-1}
\e \|\nabla u_\e \|_{L^2(\Omega_\e)} +\| u_\e\|_{L^2(\Omega_\e)} 
+ \| p_\e  \|_{L^2(\Omega_\e)} 
\le C \sqrt{\e} \, 
\left\{ \| g \|_{L^2 (\partial\Omega)}  
+  \sqrt{\e}\,  \|\nabla_{\tan}  g\|_{L^2 (\partial\Omega)}  \right\},
\end{equation}
for any $0< \e<1$.
\end{thm}

We will prove a series of lemmas before we give the proof of Theorem \ref{theorem-N}.
We begin with an estimate for $\|h\|_{H^{1/2}(\partial\Omega)}$.

\begin{lemma}\label{lemma-n-0}
Let $h$ be given by \eqref{N}. Then
\begin{equation}\label{n-00}
\| h\|_{H^{1/2}(\partial\Omega)}
\le  C \e^{-1/2} \big\{ \| g\|_{L^2 (\partial\Omega)} 
+ \e  \| \nabla_{\tan} g \|_{L^2(\partial\Omega)}     \big\}
\end{equation}
for any $0< \e<1$.
\end{lemma}

\begin{proof}
Note that 
$$
\aligned
\|h  \|_{H^{1/2}(\partial\Omega)}
&\le C \| h \|_{L^2(\partial\Omega)}^{1/2} \| h \|^{1/2} _{H^1(\partial\Omega)}\\
& \le C \left\{  \e^{-1/2} \| h  \|_{L^2(\partial\Omega)} + \e^{1/2} \| \nabla_{\tan}  h  \|_{L^2(\partial\Omega)}\right\},
\endaligned
$$
where we have used the Cauchy inequality. It is easy to see that 
$
\| h  \|_{L^2(\partial\Omega)}
\le C  \| g\|_{L^2(\partial\Omega)},
$
and
$$
\| \nabla_{\tan}  h  \|_{L^2(\partial\Omega)}
\le C\left\{  \e^{-1} \|  g\|_{L^2(\partial\Omega)}
+  \| \nabla_{\tan}  g \|_{L^2(\partial\Omega)} \right\} ,
$$
where we have used the fact $W_j =W_j(y) \in C^1(\overline{\omega}; \R^d)$.
It follows that
$$
\aligned
\| h  \|_{H^{1/2}(\partial \Omega)}
&\le C \e^{-1/2} \left\{\| g\|_{L^2(\partial\Omega)}
+ \e \|\nabla_{\tan} g \|_{L^2(\partial\Omega)}\right\}
\endaligned
$$
for any $0< \e< 1$.
\end{proof}

\begin{lemma}\label{lemma-n-1}
There exist  1-periodic functions $\phi_{ij}^\ell \in H^1_{\loc} (\R^d)$, where $1\le i, j, \ell \le d$, such that
\begin{equation}\label{phi}
\frac{\partial \phi_{ij}^\ell}{\partial y_i} = W_j^\ell (y) - K_j ^\ell
\quad \text{ and } \quad
\phi_{ij}^\ell =- \phi_{\ell j}^i, 
\end{equation}
where the index $i$ is summed from $1$ to $d$.
\end{lemma}

\begin{proof}
The proof is the same as Lemma 3.1 in \cite{KLS-2012}.
Since 
$$
\int_{Y} \left( W_j^\ell (y) - K_j^\ell\right) dy =0,
$$
one may solve the periodic boundary value problem,
$$
\left\{
\aligned
 &  \Delta f_j^\ell = W_j^\ell - K_j^\ell \quad \text{ in } Y, \\
 & f_j^\ell \text{ is 1-periodic.}
 \endaligned
 \right.
 $$
Let
$$
\phi_{ij}^\ell =\frac{\partial f_j^\ell}{\partial y_i} -\frac{\partial f_j^i}{\partial y_\ell}.
$$
Then $ \phi_{ij}^\ell =- \phi_{\ell j}^i$. Using 
$$
\frac{\partial}{\partial y_\ell} W_j^\ell = 0,
$$
we obtain the first equation in \eqref{phi}.
\end{proof}

\begin{remark}\label{remark-2}
{\rm
Using  \eqref{phi}, for $1\le j  \le d$, we may write
\begin{equation}\label{IP-1}
\aligned
n_\ell [ W_j^\ell (x/\e) -K_j^\ell]
&= \e n_\ell \frac{\partial}{\partial x_i}  \left\{ \phi_{ij}^\ell (x/ \e) \right\}\\
&= \frac{\e}{2}
\Big\{ n_\ell \frac{\partial}{\partial x_i} -n_i \frac{\partial}{\partial x_\ell } \Big\}
\left\{ \phi_{ij}^\ell (x/ \e) \right\},
\endaligned
\end{equation}
where the skew-symmetric property is used for the last step.
It follows from an  integration by parts on $\partial\Omega$ that
\begin{equation}\label{n-p}
\int_{\partial \Omega}
n_\ell [ W_j^\ell (x /\e) -K_j^\ell] \psi\, d\sigma (x) 
=-\frac{\e}{2} \int_{\partial\Omega} 
\phi_{ij}^\ell (x/ \e) 
\left\{ n_\ell \frac{\partial}{\partial x_i} -n_i \frac{\partial}{\partial x_\ell } \right\}\psi \,  d\sigma (x).
\end{equation}
This, in particular, implies that 
\begin{equation}\label{gamma-1}
\aligned
|\gamma|
 &\le C \e \|\nabla_{\tan}  g \|_{L^2(\partial\Omega)},
 \endaligned
\end{equation}
where $\gamma$ is given by \eqref{gamma}, assuming that $ (\phi_{ij}^\ell)$ are bounded.
}
\end{remark}

Let
\begin{equation}\label{Sig-1}
\Sigma_\rho =\big\{ x\in \Omega: \ \text{dist}(x, \partial\Omega)< \rho \big\}.
\end{equation}

\begin{lemma}\label{lemma-n-10}
Let $T$ be the  operator defined by 
$$
T(f) (x)=\int_{\partial\Omega} \frac{f(y)}{|x-y|^d}\, d\sigma(y).
$$
Then
\begin{equation}\label{n-10-1}
\int_{\Omega \setminus \Sigma_\e} | T(f)|^2\, dx
\le C \e^{-1} \int_{\partial\Omega} |f|^2\, d\sigma
\end{equation}
for any $0<\e<1$, and
\begin{equation}\label{n-10-2}
\int_\Omega [\text{\rm dist}(x, \partial\Omega)]^\delta |T(f)|^2\, dx
\le C_\delta  \int_{\partial\Omega} |f|^2\, d\sigma
\end{equation}
for any $\delta>1$.
\end{lemma}

\begin{proof}
It is not hard to see that
$$
\int_{\partial\Omega} \frac{d\sigma(y)}{|x-y|^d} \le \frac{C}{\text{dist}(x, \partial\Omega)}
$$
for any $x\in \Omega$.
By the Cauchy inequality,
$$
|T(f)(x)|^2
\le \frac{C}{\text{dist}(x, \partial\Omega)}
\int_{\partial\Omega}
\frac{|f(y)|^2}{|x-y|^{d}} \, d\sigma (y).
$$
The estimate \eqref{n-10-1} follows by integrating the inequality above and using Fubini's Theorem.
A similar argument gives \ref{n-10-2}.
\end{proof}

\begin{lemma}\label{lemma-n-2}
Let $(H, q)$ be a weak solution  in $H^1(\Omega;  \R^d)\times L^2(\Omega)$ of the Dirichlet problem,
\begin{equation}\label{DP-N}
\left\{
\aligned
-\Delta H +\nabla q  & =0 & \quad & \text{ in } \Omega,\\
\text{\rm div} (H) & =0 & \quad & \text{ in } \Omega,\\
H & = h & \quad & \text{ on } \partial\Omega,
\endaligned
\right.
\end{equation}
where $h$ is given by \eqref{N}.
Then
\begin{equation}\label{n-2-1}
\aligned
 \e \, \| \nabla H \|_{L^2(\Omega)}
+\| H  \|_{L^2(\Omega)}
\le C \sqrt{\e} \, \Big\{ \| g\|_{L^2(\partial\Omega)}
+ \sqrt{\e}\,  \| \nabla_{\tan} g \|_{L^2(\partial\Omega)} \Big\},
\endaligned
\end{equation}
for any $0< \e<1$.
\end{lemma}

\begin{proof}
We first point out that by the standard energy estimates  for the Stokes equations,
$$
\| \nabla H \|_{L^2(\Omega)}
\le C\,  \| h \|_{H^{1/2}(\partial\Omega)}.
$$
In view of \eqref{n-00}, this gives
$$
\e\,  \|\nabla H\|_{L^2(\Omega)}
\le C \sqrt{\e} \, \left \{ \| g\|_{L^2 (\partial\Omega)} 
+ \e\,  \| \nabla_{\tan}  g \|_{L^2 (\partial\Omega)}   \right\}.
$$
Next, we use the nontangential-maximal-function estimate,
\begin{equation}\label{n-2-3}
\| (H)^* \|_{L^2(\partial\Omega)}
\le C\,  \| h  \|_{L^2(\partial\Omega)},
\end{equation}
 to bound $H$ on $\Sigma_\e$. The estimate \eqref{n-2-3} was proved in \cite{FKV} for a Lipschitz domain
 $\Omega$, where the nontangential maximal function 
 $(H)^*$ is defined by
 \begin{equation}\label{max}
 (H)^* (x) =\sup \big\{  |H(y)|: \ y\in \Omega \text{ and } |y-x|< C_0\,  \text{dist}(y, \partial\Omega) \big\}
 \end{equation}
 for $x\in \partial\Omega$.
 It follows that
 \begin{equation}\label{n-2-4}
 \aligned
 \| H \|_{L^2(\Sigma_\e)}
 &\le C \e^{1/2}\,  \| (H)^* \|_{L^2(\partial\Omega)}
 \le C \e^{1/2} \, \| h \|_{L^2(\partial\Omega)}\\
 &\le C \e^{1/2} \, \| g\|_{L^2(\partial\Omega)}.
 \endaligned
 \end{equation}
 
 It remains to bound $H$ on $\Omega\setminus \Sigma_\e$.
 To this end, we let $(G(x, y), \Pi (x, y))$ denote the matrix of Green functions for the Stokes equation 
 \eqref{DP-N}  in $\Omega$.
That is, for each fixed 
$x\in \Omega$, $G(x, y)=(G^{ij} (x, y) )  \in H^2_{\loc}(\Omega \setminus \{ x \}; \R^{d\times d})$ and
$\Pi (x, y) =(\Pi^i (x, y))  \in L^2_{\loc} (\Omega\setminus\{ x\}; \R^d)$ satisfy 
$$
\left\{
\aligned
-\Delta_ y G^{ij} (x, y) +\frac{\partial}{\partial y_j}  \Pi^i  (x, y)  & =\delta_x \delta_{ij} 
  & \quad &  \text{ in } \Omega\setminus \{ x \},\\
\frac{\partial}{\partial y_j}  (G^{ij} (x, y)) & =0& \quad & \text{ in } \Omega \setminus \{ x \},\\
G^{ij} (x, y) & =0 & \quad & \text{ for } y \in \partial\Omega,
\endaligned
\right.
$$
in the sense of distribution.
We also require that  
$$
\Pi(x, \cdot ) \in L^1(\Omega; \R^d) \quad \text{ and } \quad \int_\Omega \Pi (x, y)\, dy  =0.
$$
Under the assumption that $\Omega$ is a bounded $C^{2,\alpha}$ domain for some $\alpha>0$,
solutions  of  the Stokes equations \eqref{DP-N}
satisfy the $C^{1, 1}$ estimate for $H$ and $C^{0, 1}$ estimate for $q$, up to 
the boundary. 
It follows that
\begin{equation}\label{G-e-1}
\aligned
|\nabla_x G(x, y)| +|\nabla_y  G(x, y)| & \le C |x-y|^{1-d},\\
|\nabla_y G(x, y)| & \le C \text{\rm dist}(x, \partial\Omega) |x-y|^{-d},\\
|\nabla^2_xG(x, y)| + |\nabla_y^2 G(x, y)| +|\nabla_x\nabla_y G(x, y)|
 & \le C |x-y|^{-d},
\endaligned
\end{equation}
and that
\begin{equation}\label{G-e-2}
\aligned
|\Pi (x, y)| &\le C |x-y|^{1-d},\\
|\Pi(x, y)-\Pi (x, z)|&\le \text{\rm dist}(x, \partial\Omega)
\{|x-y|^{-d} + |x-z|^{-d}\},\\
 |\nabla_y \Pi (x, y)| & \le C |x-y|^{-d},
\endaligned
\end{equation}
for any $x, y\in \Omega$ and $x\neq y$, $x\neq z$.
See e.g. \cite{G-Z-2019, MM-2011}.
This allows us to represent the solution $H(x)$ by
\begin{equation}\label{rep}
H^i (x)=-\int_{\partial\Omega}
\Big\{ n_k(y) \frac{\partial}{\partial y_k} G^{ij}(x, y) - \big[  \Pi^i (x, y)-\Pi^i (x, z) \big]  n_j (y) \Big\} h^j  (y) \, d\sigma (y)
\end{equation}
for any $x\in \Omega$,
where $z\in \Omega$ and $z\neq x$ (due to  the compatibility condition for $h$, the choice of $z$ is arbitrary).
Using \eqref{IP-1}, we may write $h = h^{(1)} + h^{(2)}$, where
\begin{equation}\label{n-2-6}
\aligned
h^{(1), k} 
 & =\frac{\e}{2}
\Big( n_\ell \frac{\partial}{\partial x_i} - n_i \frac{\partial }{\partial x_\ell} \Big)
\Big\{ \phi_{ij}^\ell (x/\e) g_j n_k \Big\}, \\
h^{(2), k}
&=-\frac{\e}{2} \phi_{ij}^\ell (x/\e) \Big( n_\ell \frac{\partial}{\partial x_i} - n_i \frac{\partial}{\partial x_\ell} \Big) \Big( g_j n_k\Big)
-\gamma n_k,
\endaligned
\end{equation}
for $1\le k \le d$.
Let $H^{(1)}(x)$, $H^{(2)} (x)$ be given by \eqref{rep}, with $h$ being replaced by
$h^{(1)}$, $h^{(2)}$, respectively.
Observe that by the divergence theorem,
\begin{equation}\label{IP}
\int_{\partial\Omega}
\Big( n_\ell \frac{\partial}{\partial x_i} - n_i \frac{\partial }{\partial x_\ell} \Big) v \cdot w\, d\sigma
=- \int_{\partial\Omega} v \cdot 
\Big( n_\ell \frac{\partial}{\partial x_i} - n_i \frac{\partial }{\partial x_\ell} \Big) w\, d\sigma
\end{equation}
for $1\le i, \ell\le d$.
It follows  that
$$
\aligned
& |H^{(1)}(x)|\\
 & \le C \e  \int_{\partial\Omega}
\Big\{
|\nabla_y G(x, y)| + |\nabla_y^2 G (x, y)| + |\nabla_y \Pi (x, y)|
+ |\Pi (x, y)-\Pi (x, z)| \Big\} | g(y)  |\, d\sigma (y)\\
&\le C \e 
\int_{\partial\Omega} \frac{| g(y) |}{|x-y|^d}\, d\sigma (y),
\endaligned
$$
where we have used the estimates in \eqref{G-e-1} and \eqref{G-e-2}.
In view of Lemma \ref{lemma-n-10}, we obtain 
\begin{equation}\label{n-2-8}
\aligned
\| H^{(1)} \|_{L^2(\Omega\setminus \Sigma_\e)}
 & \le C \e^{1/2}   \| g \|_{L^2(\partial\Omega)}.
 \endaligned
 \end{equation}

Finally, note that
$$
\aligned
|H^{(2)} (x)|
 & \le  C \e \|\nabla_{\tan} g\|_{L^2 (\partial\Omega)}\\
  & + C \e  \int_{\partial\Omega}
\Big\{ |\nabla_y G(x, y)| + |\Pi(x, y) -\Pi (x, z)| \Big\}
 \left( |g(y)| + |\nabla_{\tan} g (y)| \right)  d\sigma (y), 
\endaligned
$$
where we have used \eqref{gamma-1}.
Using estimates in \eqref{G-e-1} and \eqref{G-e-2}, we may deduce from \eqref{n-10-2} that 
$$
\| H^{(2)} \|_{L^2(\Omega)}
\le 
C \e \big\{ \|\nabla_{\tan}  g \|_{L^2 (\partial\Omega)}
+\| g\|_{L^2 (\partial\Omega)} \big\},
$$
which completes the proof.
\end{proof}

We are now ready to give the proof of Theorem \ref{theorem-N}

\begin{proof}[Proof of Theorem \ref{theorem-N}]

Let $(u_\e, p_\e)$ be a weak solution in $H^1(\Omega_\e; \R^d) \times L^2(\Omega_\e)$ of 
\eqref{DP-3} with $\int_{\Omega_\e} p_\e\, dx=0$, where $h$ is given by \eqref{N}.
Let $(H, q)$ be a solution of \eqref{DP-N} with boundary data $h$.
It follows from \eqref{E-14}  that
$$
\aligned
\e \|\nabla u_\e\|_{L^2(\Omega_\e)} +\| u_\e\|_{L^2(\Omega_\e)} + \| p_\e\|_{L^2(\Omega_\e)} 
& \le C \Big\{ \e \| \nabla H \|_{L^2(\Omega)} + \| H\|_{L^2(\Omega) } \Big\}\\
&\le C \sqrt{\e} 
\Big\{ \| g\|_{L^2 (\partial\Omega)} + \sqrt{\e} \|\nabla_{\tan}  g \|_{L^2(\partial\Omega)} \Big\},
\endaligned
$$
where we have used  \eqref{n-2-1} for the last inequality.
\end{proof}

As a corollary of Theorem \ref{theorem-N}, we construct a normal boundary layer corrector.

\begin{thm}\label{n-c-theorem}
Let $\Omega$ be a bounded $C^{2, \alpha}$ domain for some $\alpha>0$.
Also assume that $\partial Y_s$ is $C^{1, \alpha}$.
Let $( \Psi_n, q_n )$ be a weak solution of \eqref{DP-3} with $\int_{\Omega_\e} q_n\, dx =0$, where the boundary data $h$ is given by
\begin{equation}\label{n-c-1}
h= \Big \{- n_i W_j^i (x/\e) \Big( f_j -\frac{\partial p_0}{\partial x_j} \Big)  + b\cdot n - \gamma \Big\} n,
\end{equation}
 $p_0$ is defined by \eqref{p-0},  $b$ is the same as in Theorem \ref{t-c-theorem}, and $\gamma\in \R$ is such that $\int_{\partial\Omega} h\cdot n\, d\sigma=0$. Then
\begin{equation} \label{n-c-2}
\aligned
 & \e \|\nabla \Psi_n \|_{L^2(\Omega_\e)} +  \|\Psi_n\|_{L^2(\Omega_\e)} + \| q_n  \|_{L^2(\Omega_\e)} \\
& \qquad\qquad
\le C \sqrt{\e}
\Big\{ \| f -\nabla p_0\|_{L^2(\partial\Omega)}
+ \sqrt{\e} \|\nabla_{\tan} (f -\nabla p_0) \|_{L^2(\partial\Omega)} \Big\}
\endaligned
\end{equation}
for any $0< \e<1$. Moreover,
\begin{equation}\label{n-c-3}
|\gamma| \le C\e  \| \nabla_{\tan} ( f- \nabla p_0)\|_{L^2(\partial\Omega)}.
\end{equation}
\end{thm}

\begin{proof}

Note that by the boundary condition in  \eqref{p-0},
$$
h= \Big\{ - n_\ell \Big[ W_j^\ell (x/\e)  -K_j^\ell \Big] \Big( f_j -\frac{\partial p_0}{\partial x_j} \Big) -\gamma \Big\} n
\quad \text{ on } \partial\Omega.
$$
As a result, the estimate \eqref{n-c-2} follows readily  from Theorem \ref{theorem-N} with
$g= -(f-\nabla p_0)$.
\end{proof}


\section{Convergence rates}\label{section-C}

In this section we prove the following theorem, which contains Theorem \ref{main-theorem-1}.

\begin{thm}\label{theorem-C}
Let $\Omega$ be a bounded $C^{2, \alpha}$ domain in $\R^d$, $d\ge 2$ for some $\alpha>0$.
Also assume that $\partial Y_s$ is $C^{1, \alpha}$.
Let $(u_\e, p_\e)\in H^1(\Omega_\e; \R^d) \times L^2(\Omega_\e; \R^d)$ be a weak solution of the Dirichlet problem, 
\begin{equation}\label{DP-C}
\left\{
\aligned
-\e^2 \Delta u_\e +\nabla p_\e & = f & \quad & \text{ in } \Omega_\e,\\
\text{\rm div} (u_\e) &=0& \quad & \text{ in } \Omega_\e, \\
u_\e &=0 & \quad & \text{ on } \Gamma_\e, \\
u_\e & = b & \quad & \text{ on } \partial\Omega,
\endaligned
\right.
\end{equation}
where $f\in C^{1, 1/2} (\overline{\Omega}; \R^d)$ and $b \in H^1(\partial\Omega; \R^d)$
satisfies the compatibility condition $\int_{\partial\Omega} b\cdot n\, d\sigma =0$.
Assume that $\int_{\Omega_\e} p_\e\, dx =0$.
Then for $0< \e<1$,
\begin{equation}\label{C-rate-1}
\aligned
 & \| u_\e - W (x/\e) ( f -\nabla p_0) \|_{L^2(\Omega)}
+\| P_\e -p_0 \|_{L^2(\Omega)}\\
&\qquad
+\|\e \nabla u_\e -\nabla W(x/\e)  ( f-\nabla p_0)\|_{L^2(\Omega)}\\
& \qquad\qquad
\le C \sqrt{\e}
\Big\{ \| f\|_{C^{1, 1/2}(\Omega)}
 + \| b\cdot n\|_{H^1(\partial\Omega)}   +\| b\|_{L^2(\partial\Omega)} 
+ \e \|\nabla_{\tan} b \|_{L^2(\partial\Omega)} \Big\},
 \endaligned
\end{equation}
where $p_0$ is defined by \eqref{p-0}, $P_\e$ is given by \eqref{P},  and $C$ depends only on $\Omega$ and $Y_s$.
\end{thm}

We begin by introducing a corrector for the divergence operator.
For $1\le i, k \le d$, let $ (\chi^1 _{ik} (y), \dots, \chi^d_{ik}(y), \pi_{2, ik} (y) ) \in H^1_{\loc}(\omega; \R^d) \times L^2_{\loc}(\omega)$ be
an 1-periodic solution of
\begin{equation}\label{d-c}
\left\{
\aligned
-\Delta \chi^j _{ik} +\frac{\partial}{\partial y_j} \pi_{2, ik} & =0 & \quad&  \text{ in } \omega,\\
\frac{\partial}{\partial y_j} \chi^j _{ik} & =-W_k^i  + |Y\setminus Y_s|^{-1} K_k^i & \quad & \text{ in } \omega,\\
\chi_{ik} ^j  & =0 & \quad & \text{ on } \partial\omega.
\endaligned
\right.
\end{equation}
Since the compatibility condition, 
$$
\int_{Y\setminus Y_s} \big\{ -W_k^i  + |Y\setminus Y_s |^{-1} K_k^i \big\}\, dy=0,
$$
is satisfied, the 1-periodic solutions of \eqref{d-c} exist. 
Moreover, under the assumption that  $\partial Y_s$ is $C^{1, \alpha}$, the functions
$\nabla \chi_{ik}^j$ and $\pi_{2, ik}$ are bounded.
As usual, we extend $\chi_{ik}^j$ from $\omega$ to $\R^d$ by zero.
Fix a function $\varphi \in C_0^\infty (B(0, 1/8))$ with the properties that $\varphi\ge 0$ and
$\int_{\R^d} \varphi\, dx=1$.
Let 
\begin{equation}\label{S}
S_\e (\psi ) (x) =  \psi * \varphi_\e (x) =\int_{\R^d} \psi (y) \varphi_\e (x-y)\, dy,
\end{equation}
where $\varphi_\e (x)= \e^{-d} \varphi (x/\e)$.
Define $\Phi_\e (x) = (\Phi^1_\e (x), \Phi_\e^2 (x), \dots, \Phi_\e^d (x)) $, where 
\begin{equation}\label{Phi}
\Phi_\e^j (x) = \e  \eta_\e (x) \chi^j _{k\ell} (x/\e)  \frac{\partial}{\partial x_\ell }  S_\e \Big ( f_k -\frac{\partial p_0}{\partial x_k} \Big ),
\end{equation}
 $p_0$ is a solution of the Neumann problem \eqref{p-0}, and $\eta_\e$ is a cut-off function in $C_0^1( \Omega)$ such that 
$0\le \eta_\e\le 1$, $\eta_\e =1$ in $\Omega\setminus \Sigma_{3d \e}$, 
$\eta_\e =0$ in $\Sigma_{2d\e}$, and $|\nabla \eta_\e|\le C \e^{-1}$.
The use of the $\e$-smoothing operator $S_\e$ in \eqref{Phi}  allows us to trade  excessive  powers of $\e$ for lowering derivatives of $f-\nabla p_0$.

The following lemma will be useful to us.

\begin{lemma}\label{lemma-S}
Let $S_\e$ be defined by \eqref{S}. Then
\begin{equation}\label{S-1}
\| \psi - \eta_\e S_\e (\psi) \|_{L^2(\Omega)}
\le C \| \psi \|_{L^2(\Sigma_{3d\e})}
+ C \e \| \nabla \psi \|_{L^2(\Omega\setminus \Sigma_{d\e})}
\end{equation}
for  $0<\e<1$.
\end{lemma}

\begin{proof}
Note that
$$
\| \psi -\eta_\e S_\e(\psi)\|_{L^2(\Omega)}
\le \| (1-\eta_\e)\psi \|_{L^2(\Omega)}
+ \| \eta_\e (\psi -S_\e(\psi)\|_{L^2(\Omega)}.
$$
Clearly, the first term in the right-side hand  is bounded by $\| \psi\|_{L^2(\Sigma_{3d\e})}$.
To bound the second term, we use
$$
\psi (x) -S_\e(\psi )(x)=\int_{\R^d} \varphi_\e (y) [ \psi (x-y) -\psi (x)]\, dy
$$
and 
$$
\psi (x-y) -\psi (x) =\int_0^1  (-y) \cdot \nabla \psi (x-ty)\, dt.
$$
It follows that
$$
\aligned
 \| \eta_\e (\psi -S_\e(\psi)\|_{L^2(\Omega)}
&\le \int_{\R^d} \varphi_\e(y)  \| \psi (\cdot -y) -\psi (\cdot) \|_{L^2(\Omega\setminus \Sigma_{2d\e})}\, dy\\
& \le C\int_{\R^d} \varphi_\e (y)  |y|\, dy\,  \| \nabla \psi\|_{L^2(\Omega\setminus \Sigma_{d\e})}\\
& \le C \e   \| \nabla \psi \|_{L^2(\Omega\setminus \Sigma_{d\e})},
\endaligned
$$
where we have used Minkowski's inequality.
\end{proof}

Note that for $x\in \Omega_\e$, 
$$
\aligned
\text{\rm div} ( \Phi_\e)
&=\text{\rm div} (\chi_{k \ell} ) (x/\e) \Big[ \eta_\e \frac{\partial}{\partial x_\ell} S_\e \Big( f_k -\frac{\partial p_0}{\partial x_k} \Big) \Big]
+ \e \chi_{k\ell}^j (x/\e) \frac{\partial}{\partial x_j}
\Big[ \eta_\e \frac{\partial}{\partial x_\ell} S_\e \Big( f_k -\frac{\partial p_0}{\partial x_k} \Big) \Big]\\
&=-\Big[ W_k^\ell (x/\e) -\fint_{Y\setminus Y_s} W_k^\ell \Big]
\Big[ \eta_\e \frac{\partial}{\partial x_\ell} S_\e \Big( f_k -\frac{\partial p_0}{\partial x_k} \Big) \Big]\\
&\qquad\qquad
+ \e \chi_{k\ell}^j (x/\e) \frac{\partial}{\partial x_j}
\Big[ \eta_\e \frac{\partial}{\partial x_\ell} S_\e \Big( f_k -\frac{\partial p_0}{\partial x_k} \Big) \Big].
\endaligned
$$
Since
$$
\aligned
\text{\rm div} \Big( W(x/\e) (f -\nabla p_0) \Big)
 & = W_k^\ell (x/\e) \frac{\partial}{\partial x_\ell} \Big( f_k -\frac{\partial p_0}{\partial x_k} \Big)\\
 & = \Big[ W_k^\ell (x/\e) -\fint_{Y\setminus Y_s} W_k^\ell \Big]
 \frac{\partial}{\partial x_\ell} \Big( f_k -\frac{\partial p_0}{\partial x_k} \Big),
 \endaligned
$$
where we have used the equation in \eqref{p-0}, 
it  follows that
\begin{equation}\label{m-0}
\aligned
& \| \text{\rm div} \Big( \Phi_\e  + W(x/\e) (f -\nabla p_0)\Big)\|_{L^2(\Omega_{d\e})}\\
&\quad  \le C  \| \nabla (f-\nabla p_0)\|_{L^2(\Sigma_{3d\e})}
+ C \|   \nabla [ ( f-\nabla p_0) -S_\e (f-\nabla p_0) \big] \|_{L^2(\Omega\setminus \Sigma_{2d\e} )}\\
&\qquad\qquad
+C \e \| \nabla^2 S_\e (f -\nabla p_0)\|_{L^2(\Omega\setminus \Sigma_{2d\e})}.
\endaligned
\end{equation}

Let $(u_\e, p_\e)$ be a weak solution of \eqref{DP-C} with $\int_{\Omega_\e} p_\e\, dx =0$. Let
\begin{equation}\label{m-1}
v_\e = u_\e - \Big\{  W(x/\e) (f-\nabla p_0) +\Phi_\e + \Psi_t +\Psi_n \Big\},
\end{equation}
where $\Phi_\e$ is defined  by \eqref{Phi}, and
$\Psi_t, \Psi_n $ are given by Theorems \ref{t-c-theorem} and \ref{n-c-theorem}, respectively.
Using \eqref{W-n-2}, a direct computation shows that 
$$
\aligned
 & -\e^2 \Delta \Big\{  W(x/\e) (f-\nabla p_0)\Big\} 
+ \nabla \Big\{ p_0 + \e \pi (x/\e) (f-\nabla p_0) \Big\} \\
& =f -\e^2 \nabla \big( W(x/\e)\nabla  (f-\nabla p_0)\big)
- \e (\nabla W)(x/\e) \cdot \nabla ( f-\nabla p_0)+ \e \pi (x/\e) \nabla (f-\nabla p_0)\\
& \quad
+ \sigma_\e (f-\nabla p_0)
\endaligned
$$
in $\Omega_\e$, where $\sigma_\e$ is given by \eqref{W-n-3}.
It follows that
$$
\aligned
 & -\e^2 \Delta v_\e +\nabla \big\{ p_\e -p_0 - p_t -p_n -\e \pi (x/\e) (f-\nabla p_0) \big\}\\
&= \e^2 \Delta \Phi_\e
 + \e^2  \nabla \big( W(x/\e)\nabla  (f-\nabla p_0) \big) -\sigma_\e (f-\nabla p_0) \\
&\qquad\qquad
+\e (\nabla W)(x/\e) \cdot \nabla ( f-\nabla p_0)- \e \pi (x/\e) \nabla (f-\nabla p_0)
\endaligned
$$
in  $\Omega_\e$.
Also, observe  that 
$$
\text{\rm div}(v_\e) = -\text{\rm div} \Big( \Phi_\e + W(x/\e) (f-\nabla p_0) \Big) \quad \text{ in } \Omega_\e, 
$$
$ v_\e=0$ on $\Gamma_\e$, and that 
$$
v_\e = \gamma n \quad \text{ on } \partial\Omega,
$$
where $\gamma$ is a constant satisfying \eqref{n-c-3}.
Hence, by  Theorem \ref{theorem-E} as well as Remark \ref{remark-N} and the estimate \eqref{W-n-4} for $\sigma_\e$, 
\begin{equation}\label{C-10}
\aligned
& \e \| \nabla v_\e\|_{L^2(\Omega_\e)} + \| v_\e\|_{L^2(\Omega_\e)}\\
&\le C \Big\{
\e \|\nabla \Phi_\e \|_{L^2(\Omega_\e)}
+ \e \|\nabla ( f-\nabla p_0) \|_{L^2(\Omega)} +  \| f-\nabla p_0\|_{L^2(\Sigma_{c\e})}
+ \|\text{\rm div} (v_\e)\|_{L^2(\Omega_\e)}
+ |\gamma| \Big\}\\
& \le C \Big\{
\e \|\nabla ( f-\nabla p_0)\|_{L^2(\Omega)}
+ \| \nabla (f-\nabla p_0) \|_{L^2(\Sigma_{3d\e})}+  \| f-\nabla p_0\|_{L^2(\Sigma_{c\e})}\\
&\qquad\qquad+C \|   \nabla [ ( f-\nabla p_0) -S_\e (f-\nabla p_0) \big] \|_{L^2(\Omega\setminus \Sigma_{2d\e} )}\\
&\qquad\qquad
+ \e \| \nabla^2 S_\e (f-\nabla p_0)\|_{L^2(\Omega\setminus \Sigma_{2d\e})}
+  \e \| \nabla_{\tan} ( f-\nabla p_0)\|_{L^2(\partial\Omega)}\Big\},
\endaligned
\end{equation}
where we have used \eqref{m-0} and \eqref{n-c-3}.
Let
$$
q_\e = p_\e - p_0 - q_t -q_n -\e \pi (x/\e) (f-\nabla p_0).  
$$
Note that Theorem \ref{theorem-E} also gives
\begin{equation}\label{C-11}
\aligned
 \| q_\e  - \fint_{\Omega_\e} q_\e  \|_{L^2(\Omega_\e)}
&  \le C \Big\{
\e \|\nabla ( f-\nabla p_0)\|_{L^2(\Omega)}
+ \| \nabla (f-\nabla p_0) \|_{L^2(\Sigma_{3d\e})}  \\
&\qquad
+ \| f-\nabla p_0\|_{L^2(\Sigma_{c\e})}
\\
&\qquad  +C \|   \nabla [ ( f-\nabla p_0) -S_\e (f-\nabla p_0) \big] \|_{L^2(\Omega\setminus \Sigma_{2d\e} )}\\
&\qquad
+ \e \| \nabla^2 S_\e (f-\nabla p_0)\|_{L^2(\Omega\setminus \Sigma_{2d\e})}
+  \e \| \nabla_{\tan} ( f-\nabla p_0)\|_{L^2(\partial\Omega)}\Big\}.
\endaligned
\end{equation}

\begin{lemma}\label{lemma-C-1}
Let $(u_\e, p_\e)$ be a weak solution of \eqref{DP-C} with $\int_{\Omega_\e} p_\e\, dx=0$.
Then
\begin{equation}\label{C-1-1}
\aligned
 & \e \| \nabla \big( u_\e - W(x/\e) (f-\nabla p_0) \big) \|_{L^2(\Omega_\e)}
+ \| p_\e - p_0\|_{L^2(\Omega_\e)}\\
& \le C \e^{1/2}
\Big\{
 \e^{1/2}  \|\nabla ( f-\nabla p_0)\|_{L^2(\Omega)}
+\e^{-1/2}  \| \nabla (f-\nabla p_0) \|_{L^2(\Sigma_{3d\e})}+\e^{-1/2}  \| f-\nabla p_0\|_{L^2(\Sigma_{c\e})}\\
&\qquad\qquad\quad
+\e^{-1/2}  \|   \nabla [ ( f-\nabla p_0) -S_\e (f-\nabla p_0) \big] \|_{L^2(\Omega\setminus \Sigma_{2d\e} )}\\
&\qquad\qquad\quad
+ \e^{1/2} \| \nabla^2S_\e  (f-\nabla p_0)\|_{L^2(\Omega\setminus \Sigma_{2d\e})}
 + \| f-\nabla p_0 \|_{L^2(\partial\Omega)}\\
&\qquad\qquad\quad
+ \| b\|_{L^2(\partial\Omega)} + \e \| \nabla_{\tan}  b  \|_{L^2(\partial\Omega)}
+ \sqrt{\e} \| \nabla_{\tan} ( f-\nabla p_0) \|_{L^2(\partial\Omega)}
\Big\}
\endaligned
\end{equation}
for $0< \e< 1$.
\end{lemma}

\begin{proof}
The estimate \eqref{C-1-1} follows readily from \eqref{C-10}, \eqref{C-11},
\eqref{c-t-2}, and \eqref{n-c-2}.
\end{proof}

To bound the right-hand side of \eqref{C-1-1},
we let  $p_0=p_0^{(1)} + p_0^{(2)}$, where $p_0^{(1)}$ and $p_0^{(2)}$
are solutions of the Neumann problems,
\begin{equation}\label{p-0-1}
\left\{
\aligned
K_j^i \frac{\partial}{\partial x_i} \Big( f_j -\frac{\partial p_0^{(1)}}{\partial x_j} \Big) 
&=0 & \quad & \text{ in } \Omega,\\
n_i K_j^i \Big( f_j -\frac{\partial p_0^{(1)} }{\partial x_j}  \Big) & =0& \quad & \text{ on } \partial \Omega,
\endaligned
\right.
\end{equation}
and
\begin{equation}\label{p-0-2}
\left\{
\aligned
K_j^i \frac{\partial^2 p_0^{(2)} }{\partial x_i \partial x_j} & =0 & \quad & \text{ in }\Omega,\\
n_i K_j^i \frac{\partial p_0^{(2)}}{\partial x_j} & =- b\cdot n & \quad & \text{ on } \partial\Omega,
\endaligned
\right.
\end{equation}
respectively,
with $\int_\Omega p_0^{(1)}\, dx =\int_\Omega p_0^{(2)}\, dx =0$.

\begin{lemma}\label{lemma-C-2}
Let $p_0^{(1)}$ be a solution of \eqref{p-0-1} for some $f\in C^{1, 1/2}(\overline{\Omega}, \R^d)$.
Then 
\begin{equation}\label{C-2-0}
\aligned
\| \nabla p_0^{(1)} \|_{L^\infty(\Omega)}
+ \|\nabla^2 p_0^{(1)} \|_{L^\infty (\Omega)}
 & \le C  \| f\|_{C^{1, 1/2}(\overline{\Omega})}, \\
 \|\nabla^2 S_\e (f) \|_{L^\infty(\Omega\setminus \Sigma_{\e})}
 +  \|\nabla^2 S_\e  (\nabla p^{(1)} _0)  \|_{L^\infty(\Omega\setminus\Sigma_{\e})} 
& \le C \e^{-1/2} \| f\|_{C^{1, 1/2}(\overline{\Omega})},
\endaligned
\end{equation}
and
\begin{equation}\label{C-2-1a}
\|\nabla f-S_\e (\nabla f) \|_{L^\infty (\Omega\setminus \Sigma_{\e})}
+ \| \nabla^2 p^{(1)} _0 -  S_\e (\nabla^2 p^{(1)} _0) \|_{L^\infty(\Omega\setminus \Sigma_{\e})}
\le C \e^{1/2} \| f\|_{C^{1, 1/2}(\overline{\Omega})}.
\end{equation}
\end{lemma}

\begin{proof}
Since $\Omega$ is a bounded $C^{2, \alpha}$ domain, the first inequality in \eqref{C-2-0}  follows
from  the classical $C^2$ estimates, up to the boundary, for
second-order elliptic equations with constant coefficients.
Next, note that for $x\in \Omega\setminus \Sigma_\e$,
$$
\aligned
\frac{\partial}{\partial x_i} S_\e (\nabla f) (x)
&=\e^{-1-d}\int_{\R^d} \frac{\partial \varphi }{\partial y_i}  (y/\e) \nabla f (x-y)\, dy\\
&=\e^{-1-d}
\int_{\R^d} \frac{\partial \varphi}{\partial y_i} (y/\e)  \big[ \nabla f(x-y) -\nabla f(x)]\, dy.
\endaligned
$$
It follows that
$$
\aligned
  \|\nabla^2 S_\e(f)\|_{L^\infty(\Omega\setminus \Sigma_\e)}
& \le C \e^{-1-d} \int_{B(0, \e/4)} |\nabla \varphi (y/\e)| |y|^{1/2}\, dy\,  \| f\|_{C^{1, 1/2}(\overline{\Omega})}\\
& \le C\e^{-1/2} \| f\|_{C^{1, 1/2}(\overline{\Omega})}.
\endaligned
$$
By the interior $C^{2, 1/2}$ estimates,
$$
|\nabla^2 p^{(1)} _0(x-y) -\nabla^2 p^{(1)} _0(y)|
\le C | y|^{1/2} \Big\{ \| f\|_{C^{1, 1/2}(\overline{\Omega})} + \| p^{(1)} _0\|_{C^2(\overline{\Omega})} \Big\}
$$
for any $x\in \Omega\setminus \Sigma_{\e}$ and  $|y|\le \frac14 \e$.
As in the case of $\nabla^2 S_\e(f)$, this implies that
$$
|\nabla^2 S_\e (\nabla p_0^{(1)}) (x)|
\le C \e^{-1/2} \| f\|_{C^{1, 1/2}(\overline{\Omega})}
$$
for any $x\in \Omega\setminus \Sigma_\e$.
Finally, to see \eqref{C-2-1a}, we write
$$
 S_\e (\nabla f) (x)-\nabla f (x)   =\int_{\R^d} \varphi_ \e (x-y) [ \nabla f(x-y) -\nabla f(x)]\, dy
$$
and proceed as in the previous estimates.
\end{proof}

\begin{lemma}\label{lemma-C-3}
Let $p_0^{(2)}$  be a solution of \eqref{p-0-2}. Then
\begin{align}
\|\nabla p_0^{(2)} \|_{L^2(\Omega)}
+ \|\nabla  p_0^{(2)} \|_{L^2(\partial\Omega)}+ \e^{-1/2} \| \nabla p_0 \|_{L^2(\Sigma_{\e})}
&  \le C \| b\cdot n \|_{L^2(\partial\Omega)},  \label{C-3-a}\\
 \|\nabla^2 p_0^{(2)}\|_{L^2(\Omega)}
 + \| \nabla^2 p_0^{(2)} \|_{L^2(\partial\Omega)}
 + \e^{-1/2} \| \nabla^2 p_0^{(2)} \|_{L^2(\Sigma_{\e})} & \le C  \| b\cdot n \|_{H^1(\partial\Omega)},  \label{C-3-b}\\
\e^{-1/2} \|\nabla^2 p_0^{(2)} -S_\e (\nabla^2 p_0^{(2)} ) \|_{L^2(\Omega\setminus \Sigma_{2\e})}
 & \le C \| b\cdot n \|_{H^1(\partial\Omega)}, \label{C-3-c}\\
\e^{1/2}  \|\nabla^2 S_\e (\nabla p_0^{(2)}) \|_{L^2(\Omega\setminus \Sigma_{2\e})}
 &\le C  \| b\cdot n \|_{H^1(\partial\Omega)}, \label{C-3-d}
 \end{align}
 for $0< \e<1$.
\end{lemma}

\begin{proof}
The estimates \eqref{C-3-a}-\eqref{C-3-d} follow from the nontangential-maximal-function and square-function estimates
for the Neumann problems,
\begin{align}
\| (\nabla p_0^{(2)} )^*\|_{L^2(\partial\Omega)}
+ \left(  \int_{\Omega} \text{\rm dist}(x, \partial\Omega)
|\nabla^2 p_0^{(2)} (x)|^2\, dx \right)^{1/2}
 & \le C \| b\cdot n  \|_{L^2(\partial\Omega)},\label{C-3-0}\\
\| (\nabla^2 p_0^{(2)} )^*\|_{L^2(\partial\Omega)}
+ \left(  \int_{\Omega} \text{\rm dist}(x, \partial\Omega)
|\nabla^3 p_0^{(2)} (x)|^2\, dx \right)^{1/2}
 & \le C \| b\cdot n \|_{H^1(\partial\Omega)},\label{C-3-1}
\end{align}
where $(u)^*$ denotes the nontangential maximal function of $u$, defined by \eqref{max}. 
We remark that the estimate \eqref{C-3-0} hols if $\Omega$ is a bounded Lipschitz domain \cite{Kenig-book}, while 
\eqref{C-3-1} holds for $C^{2, \alpha}$ domains.

We only give the proof of \eqref{C-3-c}; the others follow readily from \eqref{C-3-0}-\eqref{C-3-1}.
Choose $\widetilde{\eta}_\e \in C_0^1(\Omega)$ such that $\widetilde{\eta}_\e=1$ in $\Omega\setminus \Sigma_{2\e}$,
$\widetilde{\eta}_\e=0$ in $\Sigma_{\e}$, and $|\nabla \widetilde{\eta}_\e| \le C \e^{-1}$.
Then the left-hand side of \eqref{C-3-c} is bounded by
\begin{equation}\label{C-3-2}
\e^{-1/2} \| \nabla^2 p_0^{(2)} -\widetilde{\eta}_\e S_\e ( \nabla^2 p_0^{(2)}) \|_{L^2(\Omega)}.
\end{equation}
Using the same argument as in the proof of \eqref{S-1}, we may show that \eqref{C-3-2} is bounded by
$$
C \e^{-1/2}  \| \nabla^2 p_0^{(2)} \|_{L^2(\Sigma_{3\e})}
+ C \e^{1/2} \|\nabla^3 p_0^{(2)}\|_{L^2(\Omega\setminus\Sigma_\e)}
\le C \| b\cdot n \|_{H^1(\partial\Omega)},
$$
where we have used \eqref{C-3-1} for the last step.
\end{proof}

We are now in a position to give the proof of Theorem \ref{theorem-C}.

\begin{proof}[Proof of Theorem \ref{theorem-C}]

Using Lemmas \ref{lemma-C-2} and \ref{lemma-C-3},  it is not hard to see that  the right-hand side of \eqref{C-1-1}
is bounded by
$$
C \sqrt{\e} \big\{ \| f\|_{C^{1, 1/2}({\Omega})} + \| b\cdot n \|_{H^1(\partial\Omega)}  +\| b\|_{L^2(\partial\Omega)} 
+ \e \|\nabla_{\tan} b \|_{L^2(\partial\Omega)} \big\}.
$$
As a result, we have proved that 
\begin{equation}\label{C-4-0}
\aligned
 & \e \| \nabla \big( u_\e - W (x/\e) ( f -\nabla p_0) \big) \|_{L^2(\Omega_\e)}
+\| p_\e -p_0 \|_{L^2(\Omega_\e)}\\
& \qquad\qquad
\le C \sqrt{\e}
\left\{ \| f\|_{C^{1, 1/2}(\Omega)}
 + \| b\cdot  n\|_{H^{1} (\partial\Omega)} +\| b\|_{L^2(\partial\Omega)} 
+ \e \|\nabla_{\tan} b \|_{L^2(\partial\Omega)} \right\}.
 \endaligned
\end{equation}
In view of Lemma \ref{P-lemma}, 
it remains to show that
\begin{equation}\label{C-4-1}
\| P_\e -p_0\|_{L^2(\Omega)}
 \le C \sqrt{\e}
\left\{ \| f\|_{C^{1, 1/2}(\Omega)}
 + \| b\cdot n \|_{H^{1} (\partial\Omega)}
 +\| b\|_{L^2(\partial\Omega)} 
+ \e \|\nabla_{\tan} b \|_{L^2(\partial\Omega)} \right\},
\end{equation}
where $P_\e$ is an extension of $p_\e$ to $\Omega$, defined by \eqref{P}.
To this end, we define
\begin{equation}\label{C-4-2}
p_0^\e
=\left\{
\aligned
& p_0 & \quad & \text{ if } x\in \Omega_\e,\\
& \fint_{\e (Y_f + z_k) } p_0 & \quad & \text{ if } x\in \e (Y_s + z_k) \text{ and } \e (Y+ z_k) \subset \Omega
\text{ for some } z_k \in \mathbb{Z}^d,
\endaligned
\right.
\end{equation}
i.e., we extend $p_0|_{\Omega_\e}$  to $\Omega$ in the same manner as we do $p_\e$ from $\Omega_\e$ to $\Omega$.
Then, 
$$
\aligned
\| P_\e -p_0\|_{L^2(\Omega)}
&\le \| P_\e - p_0^\e\|_{L^2(\Omega)} + \| p_0^\e - p_0\|_{L^2(\Omega)}\\
& = \| p_\e -p_0\|_{L^2(\Omega_\e)}
+ \| P_\e -p_0^\e \|_{L^2(\Omega\setminus \Omega_\e)}
+ \| p_0^\e -p_0\|_{L^2(\Omega\setminus \Omega_\e)}.
\endaligned
$$
Note that
$$
\| P_\e -p_0^\e \|_{L^2(\Omega\setminus \Omega_\e)}\le C \| p_\e -p_0\|_{L^2(\Omega_\e)}.
$$
Using Poincar\'e's inequality on each cell $\e (Y_f +z_k)$, we may show that
$$
\aligned
\| p_0^\e -p_0\|_{L^2(\Omega\setminus \Omega_\e)}
 & \le C \e \| \nabla p_0 \|_{L^2(\Omega)}
\le C \e  \left\{ \| f\|_{L^2(\Omega)} + \| b \|_{L^2(\partial\Omega)} \right\}.
\endaligned
$$
As a result, we have proved that
$$
\| P_\e -p_0\|_{L^2(\Omega)}
\le C \| p_\e -p_0\|_{L^2(\Omega_\e)}
+ C  \e  \left\{ \| f\|_{L^2(\Omega)} + \| b \|_{L^2(\partial\Omega)} \right\}.
$$
This completes the proof.
\end{proof}

\begin{remark}\label{last-r}
{\rm
Let $u(x, x/\e)$ be given by \eqref{u-0}.
Due to the discrepancy of $u_\e$ and $u(x, x/\e)$ on $\partial\Omega$,
the $O(\sqrt{\e})$ rate in Theorem \ref{main-theorem-1} is sharp. Indeed,
by  applying  the following trace  inequality to the function 
$v= v_\e=u_\e - u(x, x/\e)$,
\begin{equation}\label{last-1}
\| v\|_{L^2(\partial\Omega)}  \le C
\left\{ \e^{-1/2}  \| v\|_{L^2(\Sigma_{c\e})}  +
 \| v\|_{L^2(\Sigma_{c\e})}^{1/2} \|\nabla v\|_{L^2(\Sigma_{c\e})}^{1/2} \right\},
\end{equation}
we obtain 
$$
\aligned
\sqrt{\e} \| v_\e \|_{L^2(\partial\Omega)}
 & \le C \big\{  \| v_\e \|_{L^2(\Omega_\e)} + \e \| \nabla v_\e \|_{L^2(\Omega_\e)}\big\} \\
 & \le C  \e \| \nabla v_\e \|_{L^2(\Omega_\e)},
\endaligned
$$ 
where we have used the Cauchy inequality for the first inequality and \eqref{PI} for the second.
It follows that the error estimate
$$
\e \| \nabla v_\e \|_{L^2(\Omega_\e)} =o(\sqrt{\e}) \quad \text{ as } \e \to 0,
$$
 cannot hold in general.
 In fact, if $\Omega$ is smooth and uniformly convex, then
 \begin{equation}\label{convex}
 \aligned
 \lim_{\e\to 0} \fint_{\partial\Omega} |v_\e|^2\, d\sigma
 & =\lim_{\e \to 0} \fint_{\partial\Omega} | W(x/\e) (f-\nabla  p_0)|^2\, d\sigma\\
 &=\fint_{\partial\Omega} | K (f-\nabla p_0)|^2\, d\sigma.
 \endaligned
 \end{equation}  
 See the proof of Lemma 3.2 in \cite{Alek-2015}.
Also, note that by Theorem \ref{main-theorem-1}, 
$$
 \| \nabla v_\e \|_{L^2(\Omega_\e)}
\le C  \e^{-1/2} \| f\|_{C^{1, 1/2}(\Omega)}.
$$
This, together with \eqref{last-1}, yields
$$
\| v_\e\|_{L^2(\partial\Omega)}
\le  C \left \{ \e^{- 1/2} \| v_\e\|_{L^2(\Omega_\e)} 
+ \big( \e^{-1/2} \| v_\e\|_{^2(\Omega_\e)} \big)^{1/2} \| f\|_{C^{1, 1/2} (\Omega)}^{1/2} \right\}.
$$
As a result, it is not possible to have
$$
\|  v_\e \|_{L^2(\Omega_\e)} =o(\sqrt{\e}) \quad \text{ as } \e \to 0,
$$
unless $f=\nabla p_0$ in $\Omega$, in which case, $v_\e \equiv 0$ in $\Omega_\e$.

Finally,  to  see \eqref{last-1}, choose a function $\beta\in C^1(\R^d, \R^d)$ such that
$ \beta\cdot n \ge c_0>0$ on $\partial\Omega$, supp$(\beta) \subset \{ x\in \R^d: \text{dist}(x, \partial\Omega)\le  c \e\}$, and
$| \nabla \beta| \le C \e^{-1}$.
It follows by the divergence theorem that 
$$
\aligned
c_0 \int_{\partial\Omega} |v|^2\, d\sigma
&\le \int_{\partial\Omega} |v|^2\,  \beta \cdot n \, d\sigma
\le \int_\Omega |v|^2 \text{div} (\beta)\, dx
+  2  \int_{\Omega} |v| |\nabla v| |\beta|\, dx\\
&\le C \e^{-1} \int_{\Sigma_{c\e}} |v|^2\, dx
+  C \| v\|_{L^2(\Sigma_{c\e})} \| \nabla v\|_{L^2(\Sigma_{c\e})},
\endaligned
$$
where we have used the Cauchy inequality for the last step.
}
\end{remark}

 \bibliographystyle{amsplain}
 
\bibliography{Darcy-Law.bbl}

\bigskip

\begin{flushleft}

Zhongwei Shen,
Department of Mathematics,
University of Kentucky,
Lexington, Kentucky 40506,
USA.

E-mail: zshen2@uky.edu
\end{flushleft}

\bigskip

\end{document}